\documentclass[12pt]{imsart}

\usepackage[margin=0.9in]{geometry}
\usepackage{graphicx,mathrsfs,bbm}
\usepackage[font=large]{caption} % Set the font size to Large

\usepackage[toc]{appendix}
\RequirePackage[OT1]{fontenc}
\RequirePackage[authoryear]{natbib}
\usepackage{adjustbox}
\usepackage{amsmath,amssymb,amsthm,bm,latexsym}
\usepackage{graphicx,psfrag,epsf}
\usepackage{float,subcaption}
\usepackage{epigraph,xcolor}
\usepackage{enumerate}
\usepackage{url}
\usepackage{booktabs}
\usepackage{multirow}
\usepackage{hyperref}
\hypersetup{colorlinks=true,citecolor=blue,pdfpagemode=FullScreen, linktocpage=true}
\usepackage{cleveref}
\usepackage{amsfonts}
\usepackage{dsfont}
\usepackage{amsmath}
\usepackage{orcidlink}

\startlocaldefs
\def\E{\mathbb{E}}

\def\P{\mathbb{P}}

\newcommand{\sgn}{\mathrm{sgn}}
\newcommand{\gb}{\bar{g}}

\newcommand{\et}{e_t}
\newcommand{\ett}{e_{t+1}}
\newcommand{\tet}{\tilde{e}_t}
\newcommand{\tett}{\tilde{e}_{t+1}}
\newcommand{\tm}{t_{\mathbf{mix}}}

\renewcommand{\P}{\mathbb{P}}

\newcommand{\RNum}[1]{\uppercase\expandafter{\romannumeral #1\relax}}

\DeclareMathOperator*{\argmax}{arg\,max}

\newcommand{\tbp}{\tilde{\beta}_p}

\newtheorem{theorem}{Theorem}[section]

\newtheorem{lemma}{Lemma}[section]

\newtheorem{proposition}{Proposition}[section]

\numberwithin{equation}{section}  
\endlocaldefs

\begin{document}
	
	\begin{frontmatter}
		\title{Mixing Phases and Metastability for the Glauber Dynamics on the $p$-Spin Curie-Weiss Model}
		\runtitle{Mixing of the Glauber Dynamics on the $p$-Spin Curie-Weiss Model}
		
		\author[A.D.]{\fnms{Ramkrishna Jyoti}~\snm{Samanta}\ead[label=e1]{ramkrishna.samanta.24@ucl.ac.uk}},
		\author[S.M.]{\fnms{Somabha}~\snm{Mukherjee}\ead[label=e2]{somabha@nus.edu.sg}}
		\and
		\author[B.S.]{\fnms{Jiang}~\snm{Zhang}\ead[label=e3]{zhangj@nus.edu.sg}}
		
		\address[A.D.]{Department of Statistical Science,
			University College London\printead[presep={,\ }]{e1}}
		
		\address[S.M.]{Department of Statistics and Data Science, National University of Singapore\printead[presep={,\ }]{e2}}
		
		\address[B.S.]{Department of Statistics and Data Science, National University of Singapore\printead[presep={,\ }]{e3}}
		
		\begin{abstract}
			The Glauber dynamics for the classical $2$-spin Curie-Weiss model on $N$ nodes with inverse temperature $\beta$ and zero external field is known to mix in time $\Theta(N\log N)$ for $\beta < \frac{1}{2}$, in time $\Theta(N^{3/2})$ at $\beta = \frac{1}{2}$, and in time $\exp(\Omega(N))$ for $\beta >\frac{1}{2}$. In this paper, we consider the $p$-spin generalization of the Curie-Weiss model with an external field $h$, and identify three disjoint regions \textit{almost exhausting} the parameter space, with the corresponding Glauber dynamics exhibiting three different orders of mixing times in these regions. The construction of these disjoint regions depends on the number of local maximizers of a certain function $H_{\beta,h,p}$, and the behavior of the second derivative of $H_{\beta,h,p}$ at such a local maximizer. Specifically, we show that if $H_{\beta,h,p}$ has a unique local maximizer $m_*$ with $H_{\beta,h,p}''(m_*) < 0$ and no other stationary point, then the Glauber dynamics mixes in time $\Theta(N\log N)$, and if $H_{\beta,h,p}$ has multiple local maximizers, then the mixing time is $\exp(\Omega(N))$. Finally, if $H_{\beta,h,p}$ has a unique local maximizer $m_*$ with $H_{\beta,h,p}''(m_*) = 0$, then the mixing time is $\Theta(N^{3/2})$. We provide an explicit description of the geometry of these three different phases in the parameter space, and observe that the only portion of the parameter plane that is left out by the union of these three regions, is a one-dimensional curve, on which the function $H_{\beta,h,p}$ has a stationary inflection point. Finding out the exact order of the mixing time on this curve remains an open question. \textcolor{black}{Finally, we show that if $H_{\beta,h,p}$ has multiple local maximizers (\textit{metastable states}), then one can create a restricted version of the original Glauber dynamics, which still mixes in time $\Theta(N\log N)$.}
			
		\end{abstract}
		
		%\begin{keyword}
		%\kwd{Change-point, tightness, ergodicity, brownian-bridges, spectral-density} 
		%\end{keyword} 
		%\tableofcontents
	\end{frontmatter}
	
	\section{Introduction}
	
	The growing availability of dependent network data in modern statistics has highlighted the need for realistic and mathematically convenient approaches to model structural dependence in high-dimensional distributions. Originally developed in statistical physics to study ferromagnetism, the Ising model \citep{ising1925} has proven highly effective for modeling network dependent binary datasets occurring naturally in fields such as spatial statistics, social networks, epidemiology, computer vision, neural networks and computational biology \citep{banerjee2014, daskalakis2020, green2002, hopfield1982, montanari2010, geman1986}. The classical Ising model incorporates only pairwise interactions, which is often not enough to explain more complex higher order interactions, such as peer group effects in social networks, and multi-atomic interactions on crystal surfaces \citep{ASLANOV1988443}. A natural generalization of the classical $2$-spin Ising model, designed
	to capture higher-order dependencies, is the $p$-spin Ising model \citep{Barra2009, Derrida1980, Mukherjee2020, Mukherjee_2022}, where the quadratic interaction term in the Hamiltonian is substituted by a multilinear polynomial of degree $p\ge 2$. To be more specific, the probability mass function of this model is given by:
	\begin{equation}\label{genising}
		\P_{\beta,h,\bm J, p}(\bm x) := \frac{\exp\left\{\beta \sum_{1\le i_1,\ldots,i_p\le N} J_{i_1\ldots i_p} x_{i_1}\ldots x_{i_p} + h\sum_{i=1}^N x_i\right\}}{Z(\beta,h,\bm J,p)} \quad \text{for}~\bm x\in \mathcal{C}_N := \{-1, 1\}^N
	\end{equation}
	where $\bm J := ((J_{i_1\ldots i_p}))_{(i_1,\ldots,i_p) \in [N]^p}$ is an interaction tensor, $\beta >0$ and $h\in \mathbb{R}$ are measures of interaction and signal strengths, referred to as the inverse temperature and external magnetic field parameters, respectively, in the physics literature, and 
	$Z(\beta,h,\bm J,p)$ is the normalizing constant, needed to ensure that the probabilities in \eqref{genising} add to $1$. Ising models with higher-order interactions turn out to be useful tools in studying multi-atom interactions in lattice gas models, such as the square-lattice eightvertex model, the Ashkin-Teller model, and Suzuki’s pseudo-3D anisotropic model \citep{Barra2009, Heringa1989, Jorg2010, Ohkuwa2018, Suzuki1972, Suzuki1971, Turban2016, Yamashiro2019}. Higher-order spin systems have also appeared in statistics as a convenient way of modeling  peer-group effects in social networks \citep{Daskalakis_2020, Mukherjee_2022, Mukherjee_2021}.

	Computing the probabilities \eqref{genising} is infeasible, due to the presence of the intractable normalizing constant $Z(\beta,h,\bm J,p)$, which makes simulation from the model \eqref{genising} difficult. A classical method of simulating from the Ising model \eqref{genising} is the \textit{heat-bath Glauber dynamics}, which starts with a suitable initial configuration $\bm X^0 \in \mathcal{C}_N$, and at each step $t$, chooses a random index from the set $[N]:=\{1,\ldots,N\}$, replacing the corresponding entry of $\bm X^t$ by an observation simulated from the conditional distribution of that entry in $\bm X^t$ given the remaining ones, keeping the rest entries unchanged. The resulting stochastic process $\bm X^t$ is a Markov chain with the model \eqref{genising} as its stationary distribution, so the speed of convergence of this simulation process can be quantified by the mixing time of the Glauber dynamics. Analyzing the mixing time, however, is very difficult, unless one assumes some simplifying structural assumptions on the interaction tensor $\bm J$. One such convenient assumption is that all $p$-tuples of nodes interact with each other, and with the same intensity (taken to be $N^{1-p}$ to ensure non-trivial scaling properties), which results in the so called \textit{Curie-Weiss model}.
	
	Given parameters $\beta > 0$ and $h \in \mathbb{R}$, the $p$-tensor Curie-Weiss model is a discrete exponential family on $\mathcal{C}_N$, defined as:
	\begin{equation}\label{moddef}
		\P_{\beta, h, p}(\bm x) = \frac{1}{2^N Z(\beta, h, p)} \exp\left\{N\left(\beta \Bar{x}^p + h \Bar{x}\right)\right\}\quad\text{for}~\bm x := (x_1,\ldots,x_N) \in \mathcal{C}_N
	\end{equation}
	where $\bar{x} := \frac{1}{N} \sum_{i=1}^N x_i$, and the normalizing constant $Z(\beta, h, p)$, also referred to as the partition function,  is given by:
	\[
	Z(\beta, h, p) = \frac{1}{2^N} \sum_{\bm x \in \mathcal{C}_N} \exp\left\{ N \left(\beta \Bar{x}_N^p + h \Bar{x}_N\right)\right\}.
	\]
	The parameters $\beta$ and $h$ are typically referred to as the inverse temperature and the external magnetic field, respectively.
	The Glauber dynamics for the $p$-spin Curie-Weiss model is a Markov chain $\bm X^t$ with state space $\mathcal{C}_N$, which evolves as follows. At each step $t$, an index $I$ is chosen uniformly at random from the set $[N]$, and given $I=i$, $X_i^{t+1}$ is generated from a Rademacher distribution with mean given by $\tanh(p\beta (\bar{X}^t)^{p-1} + h)$. All the other coordinates of $\bm X^{t}$ remain unchanged in $\bm X^{t+1}$.
	Consequently, the dynamics of the mean magnetization $c_t := \bar{X}^t$ are given by the transition probabilities:
	
	\[  p\left(c, c + m\right) := \P\left(c_{t+1}=c + m\Bigg| c_t = c\right)  :=\left\{
	\begin{array}{ll}
		\left(\frac{1 - c}{2}\right) \frac{\exp({p\beta c^{p-1} + h})}{\exp({p\beta c^{p-1} + h}) + \exp({-p\beta c^{p-1} - h})} & \text{if}~m=\frac{2}{N} \\
		\left(\frac{1 + c}{2}\right) \frac{\exp({-p\beta c^{p-1} - h})}{\exp({p\beta c^{p-1} + h}) + \exp({-p\beta c^{p-1} - h})} & \text{if}~m=-\frac{2}{N} \\
	\end{array} 
	\right. \]
	and $p(c,c) = 1-p(c,c+2/N)-p(c,c-2/N)$. It can be shown that the chain $\bm X^t$ is irreducible and aperiodic, which is also reversible with respect to the stationary distribution $\P_{\beta,h,p}$. In this paper, we are interested in estimating the mixing time of this chain, defined as:
	$$\tm(\varepsilon) := \min\left\{t: \max_{\bm x\in \mathcal{C}_N} d_{\mathrm{TV}}(\P_{\bm x}(\bm X^t \in \cdot), \P_{\beta,h,p}) \le \varepsilon\right\}$$
	where $\P_{\bm x}(\bm X^t \in \cdot)$ denotes the distribution of $\bm X^t$ given $\bm X^0 = \bm x$, and $d_{\mathrm{TV}}$ stands for the total variation distance.
	
	A significant amount of work has been done on this problem for the case $p=2$, when $h=0$. \cite{Griffiths1966} showed that in the low-temperature phase ($\beta >\frac{1}{2}$), the Glauber dynamics mixes in exponential (in $N$) time. In the high-temperature phase ($\beta < \frac{1}{2}$), mixing happens in time of order $N\log N$ \citep{Aizenman1987, Bubley1997}, whereas at the critical temperature $\beta =\frac{1}{2}$, the mixing time has order $N^{3/2}$ \citep{levin2008}. Moreover, \cite{levin2008} demonstrates a cutoff phenomenon at time $[2(1-2\beta)]^{-1} N\log N$ with window size $N$ for $\beta <\frac{1}{2}$. Building on these insights, \cite{Ding2014} gave a complete characterization of the mixing time of the Glauber dynamics as $\beta$ approaches $\frac{1}{2}$, thereby illustrating its transition behavior from $\Theta(N\log N)$ through $\Theta(N^{3/2})$ to $e^{\Theta(N)}$.
	
	Moving on to Glauber dynamics on more general graphs, \cite{Hayes2005AGL} showed that any non-trivial Glauber dynamics on any bounded degree graph must have mixing time $\Omega(N\log N)$. In fact, they show that the mixing time is at least $N\log N/\Theta(d\log^2 d)$, where $d$ is the maximum degree of the graph. Later, \cite{uniform_lower_peres} established a uniform lower bound of $(\frac{1}{4} + o(1)) N \log N$ on the mixing time for the Glauber dynamics on general  ferromagnetic $2$-spin Ising models. In \cite{Chen2020RapidMO}, the authors establish upper bounds on the mixing time of the Glauber dynamics for general Ising models and general antiferromagnetic $2$-spin systems under certain high-temperature assumptions.
	Mixing time of Glauber dynamics for Ising models on random regular graphs has been analyzed in \cite{ad9501433adb4d709fec2552c50657df}, and at high temperatures for sparse Erd\H{o}s-Renyi random graphs in \cite{Mossel2007RapidMO} (also see \cite{Bianchi2008}). Upper bounds on the mixing times of the Glauber dynamics for the Potts model on general bounded-degree graphs, and for the conditional distribution of the Curie-Weiss Potts model near an equilibrium macrostate have been obtained recently in \cite{he1}. For mixing time results on more general models such as the exponential random graph models and random cluster dynamics, we refer the interested reader to \cite{Bhamidi2008MixingTO, spatial_mixing, DEMUSE2019443}

	In order to explore the mixing time of the Glauber dynamics for the $p$-spin Curie-Weiss model \eqref{moddef}, we need to introduce some notations. For $p \geq 2$ and $(\beta, h) \in \Theta := (0, \infty) \times \mathbb{R}$, define the function $H = H_{\beta, h, p} : [-1, 1] \to \mathbb{R}$ as:
	\[
	H_{\beta, h, p}(x) = \beta x^p + h x - I(x),
	\]
	where
	\[
	I(x) = \frac{1}{2} \left\{(1 + x) \log(1 + x) + (1 - x) \log(1 - x)\right\} \quad \text{for } x \in [-1, 1].
	\]
	The mixing time of the Glauber dynamics $\bm X^t$ for the $p$-spin Curie-Weiss model \eqref{moddef} exhibits phase transitions, with different orders on the following three disjoint regions of the parameter space $\Theta$ induced by the function $H_{\beta,h,p}$:
	\begin{enumerate}
		\item \textbf{$p$-locally regular points}: The point $(\beta, h)$ is said to be $p$-locally regular, if the function $H_{\beta, h, p}$ has a local maximizer $m_* = m_*(\beta, h, p) \in (-1, 1)$, satisfying $H''_{\beta, h, p}(m_*) < 0$, and has no other stationary point. The set of all $p$-locally regular points is denoted by $\mathfrak{R}_p$.
		\item \textbf{$p$-locally critical points}: The point $(\beta, h)$ is said to be $p$-locally critical, if the function $H_{\beta, h, p}$ has more than one local maximizer. The set of $p$-locally critical points is given by $\mathfrak{C}_p$. 
		\item \textbf{$p$-special points}: The point $(\beta, h)$ is said to be $p$-special if the function $H_{\beta, h, p}$ has a unique local maximizer $m_* = m_*(\beta, h, p) \in (-1, 1)$, and $H''_{\beta, h, p}(m_*) = 0$. The set of $p$-special points is given by $\mathfrak{S}_p$.  
	\end{enumerate}
	
	\begin{figure}
		\centering
		\begin{subfigure}{1.0\textwidth}
			\centering
			\includegraphics[width=\textwidth]{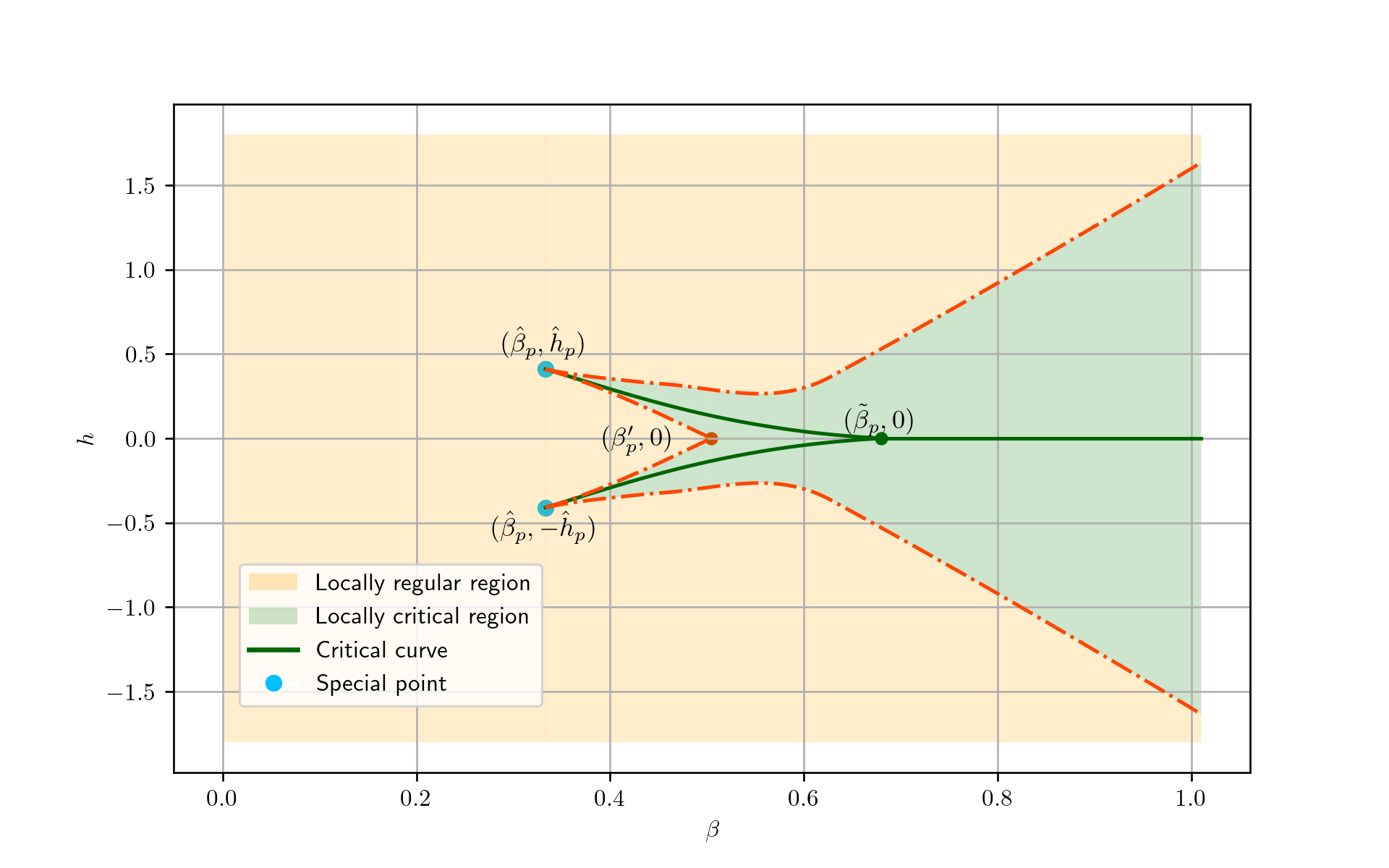}
			\caption{Mixing phase diagram for $p=4$}
			\label{fig:region even}
		\end{subfigure}
		% \hfill
		% \hspace{-0.06\textwidth} % Negative horizontal space
		\begin{subfigure}{1.0\textwidth}
			\centering
			\includegraphics[width=\textwidth]{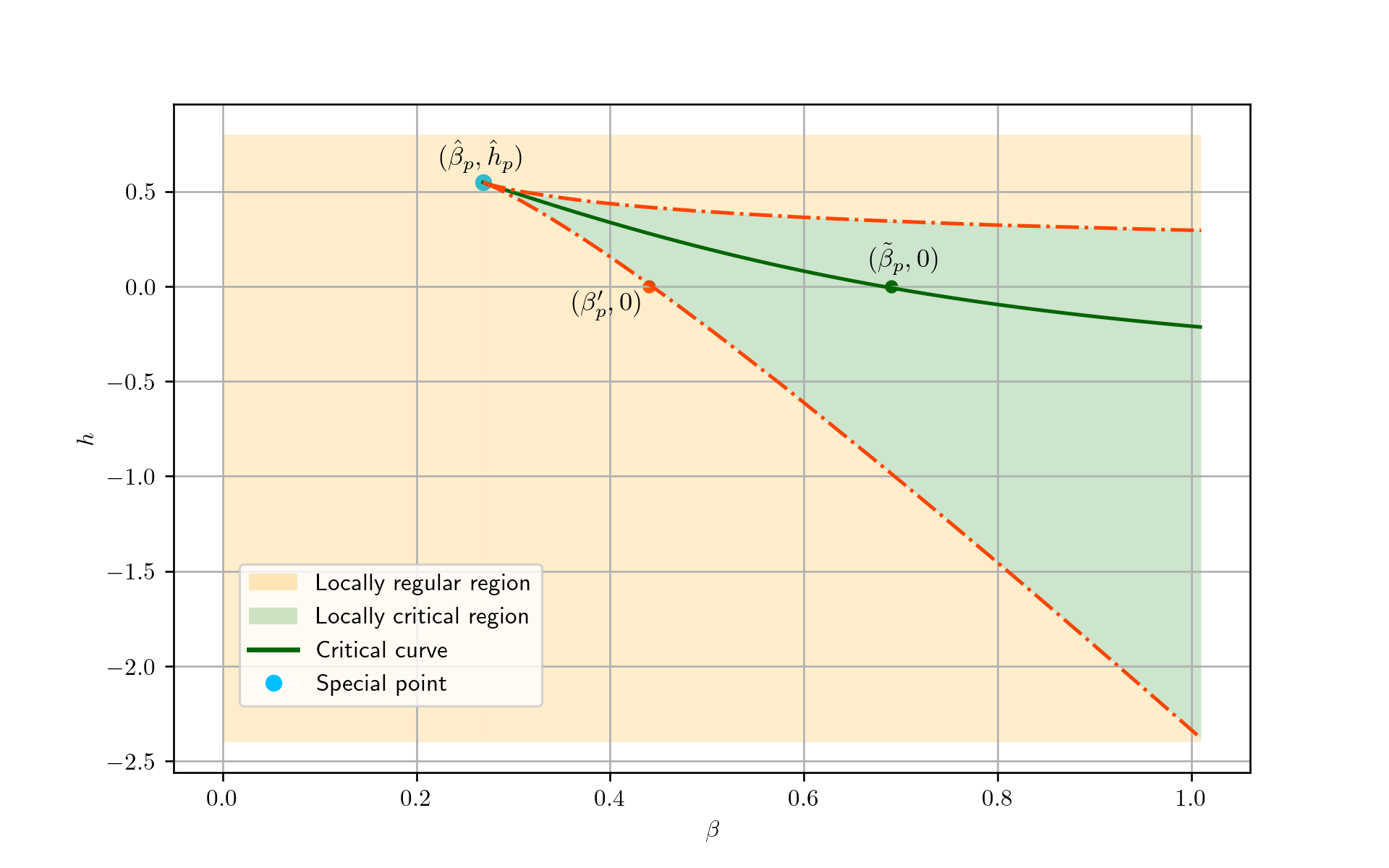}
			\caption{Mixing phase diagram for $p=5$}
			\label{fig:region odd}
		\end{subfigure}
		\caption{Mixing phase diagram }
		\label{fig:overall}
	\end{figure}

	The geometries of the regions $\mathfrak{R}_p, \mathfrak{S}_p$ and $\mathfrak{C}_p$ are made explicit in the Appendix \ref{appendix_geo}. They bear close resemblance with the $p$-regular, $p$-special and $p$-critical sets introduced in \cite{Mukherjee_2021} in the context of asymptotics of the mean magnetization and parameter estimates in the model \eqref{moddef}, which were defined in terms of uniqueness of the global maximizers of $H_{\beta,h,p}$. We will see that the set $\mathfrak{C}_p$ is a (non-uniform) band around the set of $p$-critical points in \cite{Mukherjee_2021}, which is just a one-dimensional curve in $\Theta$. The set $\mathfrak{R}_p$ is a strict subset of the set of all $p$-regular points in \cite{Mukherjee_2021}, while the notions of $p$-special points are same in both the papers. In the notations of \cite{Mukherjee_2021}, for all $p$-regular points $(\beta,h)$, the magnetization vector $\bar{X}$ simulated from the model \eqref{moddef} has an asymptotic normal distribution, for all $p$-critical points $(\beta,h)$, $\bar{X}$ is asymptotically a mixture of Gaussians, while for all $p$-special points, $\bar{X}$ is asymptotically distributed as a generalized Gaussian distribution with shape parameter $4$. Figure \ref{fig:overall} shows the three regions $\mathfrak{R}_p, \mathfrak{S}_p$ and $\mathfrak{C}_p$ for $p=4$ (in Figure \ref{fig:region even}) and $p=5$ (in Figure \ref{fig:region odd}).

	\textcolor{black}{In this paper, we show that the Glauber dynamics mixes in time $\Theta(N\log N)$ for all $(\beta,h)\in \mathfrak{R}_p$}, in time $\Theta(N^{3/2})$ for all $(\beta,h) \in \mathfrak{S}_p$, and in time $e^{\Omega(\sqrt{N})}$ for all $(\beta,h)\in \mathfrak{C}_p$. A notable difference between the phase transitions with respect to the asymptotics of the mean magnetization in the model \eqref{moddef} and the mixing times of the Glauber dynamics, is that in the former scenario, the non-global local maximizers of $H_{\beta,h,p}$ (if any) do not play any role, whereas in the latter scenario, the phase transitions are governed by all the local maximizers of $H_{\beta,h,p}$. This is because in the former case,  the non-global local maximizers do not contribute any positive mass to $\bar{X}$ asymptotically. However, it is also true that $\bar{X}$ concentrates to these local maximizers conditionally on the complement of any open set containing all the global maximizers. It is this phenomenon, that causes the Glauber dynamics to get stuck around the local maximizers for a long time, if the chain happens to initiate close to one of these. \textcolor{black}{However, even in this case, we show that it is possible to construct a restricted version of the original Glauber dynamics that still mixes fast (in time $\Theta(N\log N)$).} 
	
	Although the results in this paper are for the case $p\ge 3$ only, we would like to point out one stark difference of the phase geometry with the $p=2$ case. When $p=2$ and $h=0$, the transition thresholds for the mixing time of the Glauber dynamics and the asymptotics of $\bar{X}$ coincide at $\frac{1}{2}$, while for $p\ge 3$ and $h =0$, it will follow from Lemma \ref{pcritdescr} in Appendix \ref{appendix_geo} that the mixing time transition threshold $\beta_p'$ is strictly smaller than the transition threshold $\tilde{\beta}_p$ for the asymptotics of $\bar{X}$. This discrepancy arises due to the different natures of the function $H_{\beta,0,p}$ for $p=2$ and $p\ge 3$. To be specific, for $p=2$, the point $0$ is the only local maximizer of $H_{\beta,0,p}$ for $\beta \le \frac{1}{2}$, which guarantees fast mixing in this regime. On the other hand, for $p\ge 3$, the function $H_{\tilde{\beta}_p,0,p}$ has at least two global maximizers, which ensures the presence of multiple local maximizers of $H_{\beta,0,p}$ for values of $\beta$ slightly smaller than $\tilde{\beta}_p$ too, thereby making the mixing exponentially slow.
	
	Finally, we would like to mention that the complement of the set $\mathfrak{R}_p\bigcup \mathfrak{C}_p\bigcup \mathfrak{S}_p$ consists of all points $(\beta,h)$ where the function $H_{\beta,h,p}$ has a local maximizer and a stationary inflection point. This set is precisely the topological boundary of $\mathfrak{C}_p$ minus the $p$-special points, and is depicted by the red dashed curves in Figure \ref{fig:overall}. These stationary inflection points might hamper quick attraction of the mean magnetization chain to the maximizer of $H_{\beta,h,p}$ if it starts from the side of this inflection point opposite to the maximizer, and hence, our analysis does not go through to this boundary region. Finding the exact order of the mixing time of the Glauber dynamics on this boundary region remains an open question, although in Appendix \ref{ap:apxc} we give a crude polynomial (in $N$) lower bound to this mixing time.

	\begin{figure}
		\centering
		\includegraphics[width=1.0\linewidth]{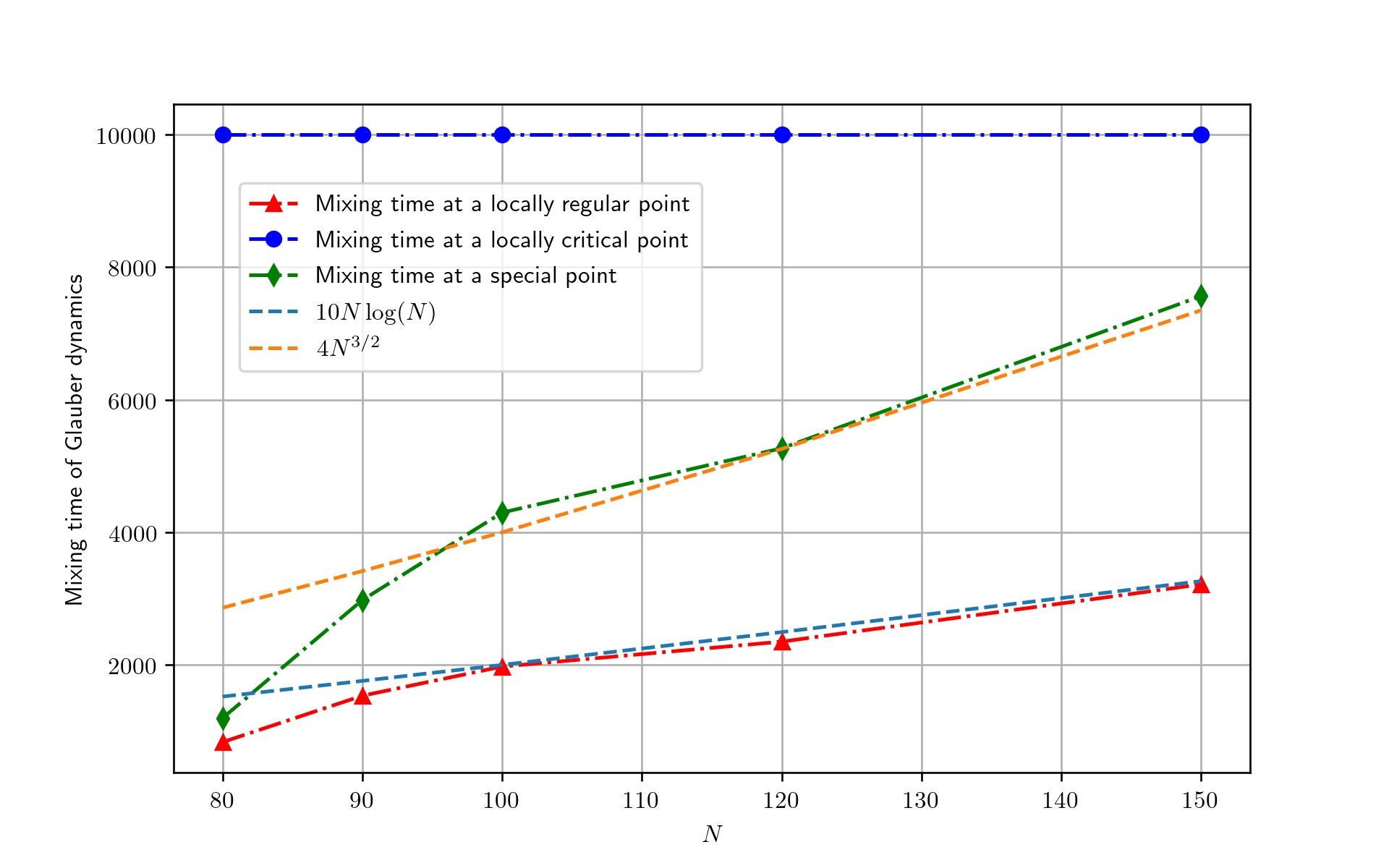}
		\caption{\textcolor{black}{Plot of the mixing times (capped at 10,000) of the Glauber dynamics against $N$ for the $4$-locally regular point \textcolor{black}{(0.054, 0.5)}, the $4$-special point \textcolor{black}{(1/3, 0.41)} and the $4$-locally critical point \textcolor{black}{(0.51, 0.184)}, compared against their theoretical estimates.} }
		\label{fig:main}
	\end{figure}
	
	\section{Main Results}
	\textcolor{black}{We now present the main results of this paper. We start with a phase transition result on the mixing time of the Glauber dynamics for the $p$-spin Curie-Weiss model.}
	\begin{theorem}\label{Theorem 1}
		For every $\epsilon \in (0,\frac{1}{2})$, $p\ge 3$ and $(\beta,h)\in \Theta$, we have the following.
		\begin{enumerate}
			\item If $(\beta,h)\in \mathfrak{R}_p$, then \( \tm(\epsilon) = \Theta_\epsilon (N \log N) \).
			\item If $(\beta,h)\in \mathfrak{C}_p$, then \( \tm(\epsilon) \ge e^{\Omega_{\epsilon}(N)} \). 
			\item If $(\beta,h)\in \mathfrak{S}_p$, then \( \tm(\epsilon) = \Theta_\epsilon(N^{3/2}) \). 
		\end{enumerate}
	\end{theorem}

	Figure~\ref{fig:main} shows simulation results for mixing times (capped at 10,000) of the Glauber dynamics in the three cases, with respect to varying $N$, for \textcolor{black}{$\epsilon = 0.35$}. The mixing time graph for the \textcolor{black}{$4$}-locally regular point \textcolor{black}{$(0.054, 0.5)$} and that for the $4$-special point \textcolor{black}{$(1/3, 0.41)$} closely approximate the graphs of the functions $10 N\log N$ and $4N^{3/2}$, respectively, whereas the mixing times for the $4$-locally critical point \textcolor{black}{$(0.51, 0.184)$} exceed $10,000$ for all values of $N \ge 80$, thereby indicating exponentially slow mixing. We give the proof of Theorem \ref{Theorem 1} in Section \ref{proofcar7}. Some parts of the proof closely follow the techiniques introduced in \cite{DEMUSE2019443} and \cite{levin2008}.
	
		\textcolor{black}{Even if the function $H_{\beta,h,p}$ happens to have a unique global maximizer with negative curvature, but possibly some non-global local maximizers too, the overall mixing of the Glauber dynamics is slowed down due to the existence of these \textit{metastable states} \citep{Bresler2022MetastableMO}, from where, the dynamics takes exponentially long time to leave. From a simulation perspective, a general advice in this case, would be to start the Glauber dynamics from an initial configuration with mean very close to the global maximizer, at least from an interval around the global maximizer where $H_{\beta,h,p}$ is strictly concave, to lower the possibility of the mean-magnetization chain getting trapped at the local maximizer pockets. This prescription is based on a phenomenon in $2$-spin Ising models with zero external field described in \cite{levin2008} (see their Theorem 3), where it is shown that for $\beta > \frac{1}{2}$, the Glauber dynamics restricted to states of non-negative mean magnetization, has mixing time $\Theta(N\log N)$. For the $p=2$ case, when $\beta >\frac{1}{2}$, the function $H_{\beta,0,2}$ has two symmetric (around $0$) global maximizers, and the restricted Glauber dynamics proceeds by flipping the signs of all entries of the candidate next step if the mean magnetization of that step happens to be negative, and accepting the step otherwise. In the context of the Curie-Weiss Potts model, \cite{he1} simply rejects the next move, should it happen to lie outside the small neighborhood of the global maximizer, and the resulting dynamics exhibits fast mixing (see their Theorem 1.8).} 
			
		\textcolor{black}{In the $p$-spin Curie-Weiss context with external field, we will focus on the largest local maximizer $m_+$ of $H_{\beta,h,p}$ and construct a restricted version of the Glauber dynamics as follows. Define:
			$$R := \left\{\bm x \in \mathcal{C}_N:  \sum_{i=1}^N x_i \ge \lceil N m_-\rceil\right\}$$
			where $m_-$ denotes the largest stationary point of $H$ not exceeding $m_+$. Start with an initial configuration $\bm x^0 \in R$. At each step $t$ of the original Glauber dynamics, if the candidate next move $\bm x^{t+1} \notin R$ but $\bm x^t \in R$, set $\bm x^{t+1} := \bm x^t$ in the restricted version, and otherwise accept $\bm x^{t+1}$ as the next move of the restricted dynamics. If $\tau_{\mathbf{mix}}(\epsilon)$ denotes the mixing time of this restricted dynamics to its stationary distribution, then:} 
	
	\textcolor{black}{\begin{theorem}\label{Theorem 1mt}
	For every $\epsilon > 0$, $\tau_{\mathbf{mix}}(\epsilon) = \Theta_\epsilon(N\log N)$. 
	\end{theorem}}
	
	\textcolor{black}{The proof of Theorem \ref{Theorem 1mt} is given in Section \ref{proofmrt}.} \textcolor{black}{Using this, one can approximately simulate from a $p$-locally critical Curie-Weiss model in time $O(N\log N)$. The first step is to run parallel restricted Glauber dynamics $\bm X^{i,\bm t}$ around each \textbf{global} maximizer $m_i$ of the function $H$ (defined in an analogous way, by restricting the original dynamics on the set $[m_i-\varepsilon, m_i+\varepsilon]$ for some suitably small $\varepsilon >0$, and rejecting steps that cross the bounderies of this interval), which, in view of Theorem \ref{Theorem 1mt}, will mix to its stationary distribution $\P_{\beta, h, p}(\cdot ~| \bar{X} \in [m_i-\varepsilon,m_i+\varepsilon]$). These parallel restricted dynamics are run for $t_N := CN\log N$ steps each, for a suitably large constant $C>0$. This is followed by simulating an independent random variable $V$ having the following distribution: $$\P(V = i) = \frac{[(m_i^2-1) H''(m_i)]^{-1/2}}{\sum_j [(m_j^2-1) H''(m_j)]^{-1/2}}~,$$ and then taking $\bm X := \bm X^{V,t_N}$. In view of Theorem 2.1 (2) in \cite{Mukherjee_2021}, $\bm X$ is very nearly a sample from the Curie-Weiss measure $\P_{\beta,h,p}$ for large $N$.}
	
	\section{Proof of Theorem \ref{Theorem 1}}\label{proofcar7}
		
		\textcolor{black}{In this section, we give the proof of Theoorem \ref{Theorem 1}, divided into the $p$-regular, $p$-critical and $p$-special cases over the next three subsections.}
		
	\subsection{Proof of Theorem \ref{Theorem 1} for $p$-locally regular points $(\beta, h)$}\label{sec:regsy}
	% We begin by establishing the burn-in period of the Glauber dynamics when the parameter \( (\beta, h) \) lies in the \textit{p-locally-regular} region. Specifically, we show that the burn-in period is \( e^{-\Omega(\sqrt{N})} \). To formalize this, we define the distance between two configurations \( X, Y \in \{-1,1\}^N \) as
	% \[
	% \rho(X, Y) = \frac{1}{2}\sum_{i=1}^{N} |X(i) - Y(i)|,
	% \]
	% which is equivalent to the Hamming distance. Using monotone coupling, we then demonstrate that for large \( N \), the expected distance between two coupled chains \( X_1 \) and \( Y_1 \) after one step satisfies the inequality
	% \[
	% \E\left( \rho(X_1, Y_1) \right) \leq e^{-\delta/N}.
	% \]
	% This contraction result serves as the foundation for proving the fast mixing time using result from \cite{LevinPeresWilmer2006}.
	
	We begin by deriving an expression for the expected drift of the mean magnetization chain of the Glauber dynamics.  
	
	\begin{lemma}\label{lemma_1}
		Let $c_t := \bar{X^t}$ denote the mean of $\bm X^t$, the $t^{\mathrm{th}}$ step of the Glauber dynamics. Then, 
		$$\E\left(c_{t+1}-c_t| c_t = c\right) = \frac{1}{N}\left(\lambda(c) - c\right),$$ where 
		\[
		\lambda(c) = \lambda_{\beta, h, p}(c) := \tanh\left(p\beta c^{p-1} +h\right)~.
		\]
	\end{lemma}
	\begin{proof}
		Note that,
		\begin{eqnarray*}
			\E\left(c_{t+1}-c_t| c_t = c\right) &=& \frac{2}{N}p(c , c + 2/N) - \frac{2}{N}p(c , c - 2/N) \\
			&=& \frac{1}{N}\cdot \frac{(1-c)e^{p\beta c^{p-1} + h} -(1+c)e^{-p\beta c^{p-1} - h}}{e^{p\beta c^{p-1} + h} + e^{-p\beta c^{p-1} - h}}\\ &=& \frac{1}{N}\left(\lambda(c) - c\right)~.
		\end{eqnarray*}
	\end{proof}
	
	Following the terminologies introduced in \cite{DEMUSE2019443}, we will call a fixed point $c^*$ of the function $\lambda$ an attractor, if $|\lambda'(c^*)|<1$. We now show that if the initial mean magnetization $c_0$ lies between an attractor $c^*$ and the fixed point of $\lambda$ nearest to $c^*$, being bounded away from both these fixed points, then the mean magnetization chain experiences a systematic drift towards the attractor $c^*$ in time linear in $N$, which is referred to as the \textit{burn-in} time in \cite{DEMUSE2019443}.

	\begin{lemma}\label{thm1}
		Suppose that $\lambda(c^*) = c^*$ and $\lambda'(c^*) < 1$. Denote by $\Bar{c}$ the smallest fixed point of $\lambda$ greater than $c^*$ (in case no such $\Bar{c}$ exists, take $\Bar{c}=1$). Suppose that the initial mean magnetization $c_0$ satisfies $c^*+\delta < c_0 <\Bar{c}-\delta$ for some fixed $\delta>0$. Then, there exists $k>0$ depending only on $\delta, p, \beta, h$ such that 
		$$\P(c_{\lceil k N\rceil } < c_0-\alpha)\geq 1-e^{-\Omega (\sqrt{N})}$$
		where $\alpha>0$ is any fixed real number, satisfying $[c_0-2\alpha,c_0+\alpha]\subseteq [c^*+\delta,\bar{c}-\delta]$.
	\end{lemma}
	\begin{proof}
		Since $\lambda'(c^*) < 1$, the function $\xi(c) := \lambda(c)-c$ must be strictly decreasing on $[c^*,c^*+\nu]$ for some $\nu > 0$, which implies that $\xi(c)<0$ for all $c\in (c^*,c^*+\nu]$.  Since $\xi$ cannot change sign on the interval $(c^*,\bar{c})$, it must be true that $\lambda(c)<c$ for all $c\in (c^*,\bar{c})$. Define:
		
		$$\gamma := \inf_{c \in [c^* + \delta, \bar{c} - \delta]} \frac{c-\lambda(c)}{2} >0.$$ Recall that $\alpha > 0$ is a fixed real number, satisfying $[c_0 - 2\alpha, c_0 + \alpha] \subseteq [c^* + \delta, \bar{c} - \delta]$. By Lemma \ref{lemma_1}, on the event $D_t = D_t(\alpha) := \{c_0 - 2\alpha \le c_t \le c_0 + \alpha\}$, we thus have:
		
		\[
		\mathbb{E}[c_{t+1} - c_t \mid c_t] < -\frac{\gamma}{N}~.
		\]
		Next, for all non-negative integers $t_1 \le t_2$, define:
		$$S_{t_1, t_2} := \sum_{t=t_1+1}^{t_2}(c_t-c_{t-1}+\frac{\gamma}{2N})\mathbbm{1}_{D_{t-1}}.$$
		Define $\mathcal{F}_t := \sigma(\bm X^0,\ldots,\bm X^t)$, which is the natural filtration for the Glauber dynamics. Then, we have 
		\begin{align*}
			\E[e^{\theta S_{t_1, t_2}}] &= \E[e^{\theta S_{t_1, t_2-1}}\E(e^{\theta(c_{t_2}-c_{t_2-1}+\frac{\gamma}{2N})\mathbbm{1}_{D_{t_2-1}}}|\mathcal{F}_{t_2-1})] \\
			&= \E\left[e^{\theta S_{t_1, t_2-1}}\left( 1 + \E\left[\theta\left(c_{t_2}-c_{t_2-1}+\frac{\gamma}{2N}\right)\mathbbm{1}_{D_{t_2-1}}\Big|\mathcal{F}_{t_2-1}\right] + O\left(\frac{\theta^2}{N^2}\right)\mathbbm{1}_{D_{t_2-1}}\right)\right]\\
			& \leq \E\left[e^{\theta S_{t_1, t_2-1}}\left( 1 - \frac{\gamma\theta}{2N}\mathbbm{1}_{D_{t_2-1}} + O\left(\frac{\theta^2}{N^2}\right)\mathbbm{1}_{D_{t_2-1}}\right)\right].
		\end{align*}
		Hence, by taking $\theta = c\sqrt{N}$ for some small constant $c>0$ and large $N$, we thus have: 
		$$\E(e^{\theta S_{t_1,t_2}})\leq \E(e^{\theta S_{t_1,t_2-1}})\leq \ldots \leq E(e^{\theta S_{t_1,t_1}}) = 1.$$ Therefore, by Chernoff bound,
		\begin{equation}\label{upd}
			\P(S_{t_1, t_2} \geq \alpha/2) \leq \frac{\E(e^{c\sqrt{N} S_{t_1, t_2}})}{e^{c\sqrt{N}\alpha/2}} = e^{-\Omega(\sqrt{N})}~.
		\end{equation}
		Now, define the event:
		$$B_{t_1,t_2} = B_{t_1,t_2}(\alpha) := \left(\bigcap_{t_1\leq t< t_2}D_t(\alpha)\right)\bigcap \left\{c_{t_2}-c_{t_1} > \frac{\alpha}{2}\right\}.$$
		On the event $\cap_{t_1\leq t< t_2}D_t$, we have $S_{t_1, t_2} = c_{t_2} - c_{t_1} + \frac{\gamma}{2N}(t_2-t_1)$
		and hence, $$B_{t_1, t_2}(\alpha) \subseteq \{S_{t_1,t_2} \geq \alpha/2\}.$$ Fix a sequence $T=O(N)$, whence it follows from \eqref{upd} that:
		\begin{equation}\label{llpd}
			\P\left(\bigcup_{0\leq t_1< t_2 \leq T}B_{t_1, t_2}\right)\leq e^{-\Omega(\sqrt{N})}~.
		\end{equation}
		Now, suppose that $c_t > c_0 +\alpha$ for some $0\le t\le T$.
		Define: $$t_1 := \max\{0\le s< t: c_s \le c_0\}\quad\text{and}\quad t_2 := \min \{t_1<s\le T: c_s >c_0+\alpha\}.$$
		Since $c_{t_1} \le c_0$ and $|c_s-c_{s-1}|\le 2/N$ for all $N$, we must have $t_1 <t- N\alpha/2$. Hence, for all large $N$, $t_2$ exists. Also, since $c_{t_2}>c_0+\alpha$, one must have $t_2>t_1+N\alpha/2$. Note that for all large $N$, $$c_0 < c_s \le c_0+\alpha \quad\text{for all} \quad t_1 < s < t_2\quad\quad\text{and} \quad \quad c_{t_2-1}-c_{t_1+1} >\alpha-\frac{4}{N} >\frac{\alpha}{2}~.$$ 
		Hence, for all large $N$,
		$$\bigcup_{t=0}^T \{c_t >c_0+\alpha\} \subseteq \bigcup_{0\leq t_1< t_2 \leq T}B_{t_1, t_2}$$
		and hence, it follows from \eqref{llpd} that:
		\begin{equation}\label{lnpg}
			\P\left( c_t>c_0 + \alpha \text{ for some } 0\leq t \leq T \right)\leq e^{-\Omega(\sqrt{N})}.
		\end{equation}
		This step thus shows that the mean magnetization chain stays away from $\Bar{c}$ before time $T = O(N)$ with high probability. Next, note that:
		\begin{equation}\label{kl1}
			\P\left(c_t\ge c_0 - 2\alpha \text{ for all } 0\leq t \leq T \right) \leq \P\left(\bigcap_{0\leq t\leq T}D_t(\alpha)\right) + \P\left( c_t>c_0 + \alpha \text{ for some } 0\leq t \leq T \right).
		\end{equation}
		Now, since on the event $\cap_{0\leq t\leq T}D_t$, $$S_{0,T} = \sum_{t=1}^{T}\left(c_t - c_{t-1} + \frac{\gamma}{2N}\right) = c_T - c_0 + \frac{\gamma}{2N}T \geq -2 + \frac{\gamma}{2N}T~,$$ we have:
		\begin{equation}\label{kl2}
			\bigcap_{0\leq t\leq T}D_t\subseteq \left\{ S_{0,T}\geq -2 + \frac{\gamma}{2N}T\right\}.
		\end{equation} 
		Now, by a Chernoff bound, we have:
		\begin{equation}\label{kl3}
			\P\left(S_{0,T}\geq -2 + \frac{\gamma}{2N}T\right) \le \frac{\E \exp[\theta S_{0,T}]}{\exp\left[\theta \left(-2+\frac{\gamma}{2N}T\right)\right]} =e^{-\Omega(\sqrt{N})}
		\end{equation}
		if $T=\lceil k N\rceil$ with $k >5/\gamma$. Combining \eqref{lnpg}, \eqref{kl1}, \eqref{kl2} and \eqref{kl3}, we have:
		\begin{equation}\label{kl4}
			\P\left(c_t\ge c_0 - 2\alpha \text{ for all } 0\leq t \leq T \right) \le e^{-\Omega(\sqrt{N})}.
		\end{equation}
		Next, note that by an exactly similar argument as in \eqref{lnpg}, we have:
		\begin{equation}\label{lnpd}
			\P(c_T > c_t+\alpha~\text{for some}~0\le t\le T) \le e^{-\Omega(\sqrt{N})}
		\end{equation}
		Combining \eqref{kl4} and \eqref{lnpd}, we have:
		\begin{eqnarray*}
			\P(c_T \geq c_0 - \alpha) &\leq&  \P\left(c_T\geq c_0 - \alpha \text{ and } c_t<c_0 - 2\alpha \text{ for some } 0\leq t \leq T \right)\\ &+& \P(c_t\geq c_0-2\alpha \text{ for all } 0\leq t \leq T) \leq e^{-\Omega(\sqrt{N})}. 
		\end{eqnarray*}
		The proof of Lemma \ref{thm1} is now complete.
	\end{proof}
	
	By applying Lemma \ref{thm1} repeatedly, we can show that after each burn-in epoch (taking time $O(N)$), the mean magnetization chain moves closer to the attractor $c^*$ with exponentially high probability. This will ensure that the chain visits every small neighborhood of $c^*$ after a finite number of burn-in epochs. The next lemma shows that in each of these neighborhoods, it spends at least an exponentially long time with exponentially high probability. 
	
	\begin{lemma}\label{corollary_reg}
		Suppose that $c^*$ is the unique fixed point of $\lambda$, and $\lambda^{'}(c^*)<1.$ Then, for any $\varepsilon >0,$ there exists $k,y >0$ such that for any initial configuration with mean magnetization $c_0$, 
		
		$$\P\left(\bigcap_{t=kN}^{e^{y\sqrt{N}}} \{c_*-\varepsilon <c_t<c^*+\varepsilon\}\right) \ge 1-e^{-\Omega(\sqrt{N})}.$$
	\end{lemma}
	
	\begin{proof}
		We only prove Lemma \ref{corollary_reg} for the case $c_0>c_*$, as the other case is handled similarly. By repeated applications of Lemma \ref{thm1} and using the strong Markov property, we conclude that the mean magnetization chain visits the neighborhood $[-1,c^*+\varepsilon/2)$ within time $O(N)$ (at the end of finitely many burn-in epochs) with probability at least $1-e^{-\Omega(\sqrt{N})}$. Let $T$ denote the first time of this visit. Then, there exists $\ell>0$ such that $\P(T \le \ell N) \ge 1-e^{-\Omega(\sqrt{N})}$. Note that $c_{T} \ge c_*+\varepsilon/2-2/N$ and hence, we can take $N$ large enough, so that $c_{T} \ge c_*+\varepsilon/3$.
		
		Now, the proof of Lemma \ref{thm1} coupled with the strong Markov property, implies that for any $\alpha >0$ and time $T'$,
		$$\bigcup_{t=0}^{T'} \{c_{T+t} >c_T+\alpha\} \subseteq \bigcup_{0\leq t_1< t_2 \leq T'}B_{T+t_1, T+t_2}(\alpha)~.$$ Taking $\alpha := \varepsilon/7$, we once again conclude from the proof of Lemma \ref{thm1} that there exists exists a constant $\phi := \phi_\varepsilon >0$, such that for every $0\leq t_1<t_2 \leq T'$, 
		$\P(B_{T+t_1,T+t_2}) \le e^{-\phi\sqrt{N}}$, and hence,
		$$\P\left(\bigcup_{0\leq t_1< t_2 \leq T'}B_{T+t_1, T+t_2}\right)\leq T'^2 e^{-\phi \sqrt{N}} \le e^{-\Omega(\sqrt{N})}$$
		if $T' = e^{z\sqrt{N}}$ with $z<\phi/2$. Hence, we have:
		$$\P\left(\bigcup_{t=0}^{e^{z\sqrt{N}}} \{c_{T+t} \ge c^* +\varepsilon\}\right) \le \P\left(\bigcup_{t=0}^{e^{z\sqrt{N}}} \left\{c_{T+t} > c_T +\frac{\varepsilon}{7}\right\}\right) \le e^{-\Omega(\sqrt{N})}~.$$
		Hence, we have:
		\begin{eqnarray*}
			\P\left(\bigcup_{t=\ell N}^{e^{z\sqrt{N}}} \{c_{t} \ge c^* +\varepsilon\}\right) &\le& \P\left(T \le \ell N~,~\bigcup_{t=\ell N}^{e^{z\sqrt{N}}} \{c_{t} \ge c^* +\varepsilon\}\right) + \P(T>\ell N)\\ &\le& \P\left(\bigcup_{t=T}^{T+e^{z\sqrt{N}}} \{c_{t} \ge c^* +\varepsilon\}\right) + \P(T>\ell N) = e^{-\Omega(\sqrt{N})}.
		\end{eqnarray*}
		By an exactly similar argument, we can conclude that there exists $m,w>0$, such that:
		$$ \P\left(\bigcup_{t= mN}^{e^{w\sqrt{N}}} \{c_{t} \le c^* -\varepsilon\}\right) \le e^{-\Omega(\sqrt{N})}.$$ Lemma \ref{corollary_reg} now follows on taking $k := \max\{\ell,m\}$ and $y=\min \{z,w\}$.
	\end{proof}

	We now construct a coupling of the Glauber dynamics $\bm X^t$ with another chain $\bm Y^t$ with the same transition probability structure, but with possibly different starting configurations. Towards this, suppose that $\bm x$ and $\bm y$ are two spin configurations, and set $\bm X^0=\bm x$ and $\bm Y^0 = \bm y$. Now, inductively assume that the joint law of $(\bm X^0, \bm Y^0),\ldots,(\bm X^t,\bm Y^t)$ has been defined for some $t$. Given this joint law, generate $U_{t+1}$ uniformly from $[0,1]$ (independent of all the random variables considered so far), followed by selecting a vertex $j$ uniformly at random, and defining
	\[
	X_j^{t+1} = \begin{cases} 
		1 & \text{if } U_{t+1} \leq f(\bar{X}^t)\\
		-1 & \text{otherwise}
	\end{cases}
	\quad\quad \text{and} \quad\quad
	Y_j^{t+1} = \begin{cases} 
		1 & \text{if } U_{t+1} \leq f(\bar{Y}^t)\\
		-1 & \text{otherwise}
	\end{cases}
	\]
	keeping all the other coordinates of $\bm X^{t+1}$ and $\bm Y^{t+1}$ unchanged from those of $\bm X^{t}$ and $\bm Y^{t}$, respectively, where:
	$$f(t) := \frac{1+\tanh(p\beta t^{p-1}+h)}{2}~.$$
	A crucial step towards analyzing the mixing time of the Glauber dynamics is to estimate the expected Hamming distance ($\rho$) between $\bm X^t$ and $\bm Y^t$. Towards this, we start with the case $t=1$, under the much simpler assumption that the initial configurations are differing at just one coordinate.
	
	\begin{lemma}\label{lemma_reg}
		Assume that $\sup_{c\in [-1,1]}|\lambda^{'}(c)|<1$. Then, there exists $\delta >0$ depending only on $\beta,h,p$, such that for every possible pair $(\bm x,\bm y)$ of starting configurations of the the coupled chain $(\bm X^t,\bm Y^t)$, we have:
		$$\E \rho(\bm X^t, \bm Y^t)\leq \rho(\bm x, \bm y) e^{-t\delta/N}\quad\quad\text{for all large}~N.$$
	\end{lemma}
	
	\begin{proof}
		Suppose that $\rho(\bm x,\bm y) = k$, i.e. $\bm x$ and $\bm y$ are differing at exactly $k$ coordinates. If the randomly chosen index $j$ in the first step of the Glauber dynamics happens to be one of these $k$ indices, then $\rho(\bm X^1,\bm Y^1)$ would be $k$ with probability $\xi := |f(\bar{x})-f(\bar{y})|$ and $k-1$ with probability $1-\xi$. On the other hand, if $j$ happens to be one of those indices at which $\bm x$ and $\bm y$ agree, then $\rho(\bm X^1,\bm Y^1)$ would be $k+1$ with probability $\xi$ and $k$ with probability $1-\xi$. Hence, 
		\begin{eqnarray*}
			\mathbb{E} \rho(\bm X^1, \bm Y^1) &=& \frac{k}{N}\left(k\xi + (k-1)(1-\xi)\right) + \left(1-\frac{k}{N}\right)\left((k+1)\xi + k(1-\xi)\right)\\ &=& k+\xi - \frac{k}{N} \\&\le & k\left(1- \frac{1-\sup_{c\in [-1,1]}~|\lambda'(c)|}{N}\right)\\&\le & k \exp\left(-\frac{1-\sup_{c\in [-1,1]}|\lambda'(c)|}{N}\right)
		\end{eqnarray*}
		This proves Lemma \ref{lemma_reg} for $t=1$, with $\delta := 1-\sup_{c\in [-1,1]}|\lambda'(c)|$. What this implies, is that at any step $t$, we have:
		$$\E \left[\rho(\bm X^t,\bm Y^t)\Big| \bm X^{t-1}, \bm Y^{t-1}\right] \le \rho(\bm X^{t-1},\bm Y^{t-1}) e^{-\delta/N} ~\implies~ \E \rho(\bm X^t,\bm Y^t) \le e^{-\delta/N} \E  \rho(\bm X^{t-1},\bm Y^{t-1})~.$$ The above recursion inequality gives:
		$$\E \rho(\bm X^t,\bm Y^t)  \le e^{-t\delta/N} \rho(\bm x,\bm y)~.$$
		This completes the proof of Lemma \ref{lemma_reg}.
		% Note that if $j=i$, then $\rho(\bm X^1,\bm Y^1)$ is $2$ with probability $|f(\bar{x})-f(\bar{y})| = O(1/N)$, and $0$ with the remaining probability. On the other hand, if $j\ne i$, then $\rho(\bm X^1,\bm Y^1)$ is $2$ with probability $|f(\bar{x})-f(\bar{y})|$ and $1$ with the remaining probability. Consequently, we have:
		% \begin{eqnarray*}
			%     \mathbb{E} \rho(\bm X^1, \bm Y^1) &=& \frac{N-1}{N}\left(1+|f(\bar{x})-f(\bar{y})|\right) + O\left(\frac{1}{N^2}\right)\\ &\le& \frac{N-1}{N} \left(1+ \frac{1}{N}\sup_{-1\le c\le 1}|\lambda'(c)|\right) + O\left(\frac{1}{N^2}\right)\\&\le & 1-\frac{1-\sup_{-1\le c\le 1}|\lambda'(c)|}{N} + O\left(\frac{1}{N^2}\right)\\ &\le &\exp\left(-\frac{1-\sup_{-1\le c\le 1}|\lambda'(c)|}{N}\right)+ O\left(\frac{1}{N^2}\right).
			% \end{eqnarray*}
		% The proof of Lemma \ref{lemma_reg} is now complete on taking $\delta := \left(1-\sup_{-1\le c\le 1}|\lambda'(c)|\right)/2$.
	\end{proof}
	
	% \begin{corollary}\label{localderv}
		%      Suppose that $\sup_{-1\leq c\leq 1}|\lambda^{'}(c)|<1$. Then, for every possible pair $(\bm x,\bm y)$ of starting configurations of the coupled chain $(\bm X^t,\bm Y^t)$, we have: 
		%      $$\E\rho(\bm X^t,\bm Y^t) \le\rho(\bm x,\bm y) e^{-t\delta/N}\quad\quad\text{for all large}~N.$$
		% \end{corollary}
	% \begin{proof}
		%    First, note that if $\bm x$ and $\bm y$ are the actual initial configurations differing at $k$ places, then we can choose a sequence of configurations $\bm x_0,\ldots,\bm x_k$ such that $\bm x_0 = \bm x$, $\bm x_k = \bm y$, and $\rho(\bm x_j,\bm x_{j+1}) = 1$ for all $0\le j<k$. Then, applying Lemma \ref{lemma_reg} on each of these neighboring pairs of initial configurations (formally, by constructing a $(k+1)$-dimensional coupling of the Glauber dynamics with these initial states, and using the same uniform noise for all the $k+1$ chains at each step), we get by triangle inequality,
		%     $$\E \rho(\bm X^1,\bm Y^1) \le k e^{-\delta/N}.$$
		%     This implies that at any step $t$, we have:
		%     $$\E \left[\rho(\bm X^t,\bm Y^t)\Big| \bm X^{t-1}, \bm Y^{t-1}\right] \le \rho(\bm X^{t-1},\bm Y^{t-1}) e^{-\delta/N} ~\implies~ \E \rho(\bm X^t,\bm Y^t) \le e^{-\delta/N} \E  \rho(\bm X^{t-1},\bm Y^{t-1})~.$$ The above recursion inequality gives:
		%     $$\E \rho(\bm X^t,\bm Y^t)  \le e^{-t\delta/N} \rho(\bm x,\bm y)~.$$
		%    This completes the proof of Corollary \ref{localderv}.
		% \end{proof}
	
	Note that the assumption $\sup_{c\in [-1,1]} |\lambda'(c)| < 1$ in Lemma \ref{lemma_reg} is too strong, and should preferably be replaced by the more realistic assumption $|\lambda'(c^*)|<1$, where $c^*$ is the unique fixed point of $\lambda$. 
	In fact, the latter assumption, in view of continuity of $\lambda'$, implies that $\sup_{c\in [c^*-\varepsilon,c^*+\varepsilon]} |\lambda'(c)|<1$ for some $\varepsilon > 0$, and by Lemma \ref{corollary_reg}, we know that with probability at least $1-e^{-\Omega(\sqrt{N})}$, after $O(N)$ time, both the mean magnetization chains corresponding to $\bm X^t$ and $\bm Y^t$ enter the interval $[c^*-\varepsilon,c^*+\varepsilon]$ and stay there for at least another $e^{\Omega(\sqrt{N})}$ amount of time. We wait throughout these initial burn-in epochs for this to happen at the cost of spending $O(N)$ time. Then, we completely forget about the history of the chains, thereby starting afresh once the mean magnetization chains already enter the interval $[c^*-\varepsilon,c^*+\varepsilon]$ and stays there for at least another $e^{\Omega(\sqrt{N})}$ amount of time. In this sense, Lemma \ref{lemma_reg} still holds under the milder assumption $|\lambda'(c^*)|<1$ for the unique fixed point $c^*$ of $\lambda$, at least for all $t \le e^{\kappa\sqrt{N}}$ for some $\kappa > 0$.
	
	By Lemma \ref{lemma_reg}, if $|\lambda'(c^*)| <1$, then $\E \rho(\bm X^t, \bm Y^t) \le N e^{-t\delta/N}$ for all large $N$, and hence, this expected Hamming distance can be controlled below any threshold $\epsilon>0$ as long as $t>(N/\delta)\log(N/\epsilon)$. In fact, we can formally apply Theorem 14.6 and Corollary 14.7 of \cite{LevinPeresWilmer2006} to conclude that:
	$$\tm(\epsilon) \le \left\lceil \frac{N}{\delta}\log\frac{N}{\epsilon}\right\rceil = O_\epsilon(N\log N)~.$$
	Note that the initial $O(N)$ time of burn-in does not affect this result, as $O(N\log N) + O(N)$ is still $O(N\log N)$. In view of Lemma \ref{regcht1}, we therefore just need to handle the case $\lambda'(c^*) \le -1$ in order to complete the proof of the upper bound part of Theorem \ref{Theorem 1} for locally regular points. Below, we give a proof for this case, that in fact encompasses the entire regime $\lambda'(c^*) <1$.
	%This, in view of Lemma \ref{regcht1}, completes the proof of the upper bound part of Theorem \ref{Theorem 1} for locally regular points.

	{\color{black} Define $e_t := c_t - c^*$, and suppose that $\varepsilon > 0$ be such that $\delta := \sup_{x\in [c^*-\varepsilon, c^* + \varepsilon]}\lambda^{'}(x)<1.$  Assume that $e_0 > \mu$ for some fixed $\mu > 0$, and define $\tau_0 = \text{min}\{t\geq 0: e_t\leq 1/N\}.$ 
		It follows from Lemma \ref{corollary_reg} that with probability $1-e^{-\Omega(\sqrt{N})}$, the chain $e_t$ remains in the interval $(-\varepsilon, \varepsilon)$ from time $O(N)$ till time $e^{\Omega(\sqrt{N})}$.
		Once the chain $e_t$ enters the interval $(-\varepsilon,\varepsilon)$ in time $O(N)$, we have:
		\begin{align*}
			\E((e_{t+1} - e_t) \mathbbm{1}_{\{\tau_0 > t\}}|\mathscr{A}_t) &= \frac{1}{N}(\lambda(e_t + c^*) - e_t - c^*)\mathbbm{1}_{\{\tau_0 > t\}}\\ &= \frac{1}{N}(\lambda(c^*) + e_t\lambda^{'}(\xi) - e_t - c^*)\mathbbm{1}_{\{\tau_0 > t\}}\\
			&\leq \frac{1}{N}(\delta -1)e_t \mathbbm{1}_{\{\tau_0 > t\}},\\
		\end{align*}
		where $\mathscr{A}_t := \sigma(e_0,\ldots,e_t)$ and $\xi$ lies between $c^*$ and $c_t$. Hence, 
		\begin{align*}
			\E(e_{t+1}\mathbbm{1}_{\{\tau_0 > t+1\}}|\mathscr{A}_t) \leq \E(e_{t+1}\mathbbm{1}_{\{\tau_0 > t\}}|\mathscr{A}_t)  = e_t\mathbbm{1}_{\{\tau_0 > t\}} \left(1 - \frac{1-\delta}{N}\right)
		\end{align*}
		Hence by iterating , we have:
		\begin{align*}
			\E(e_{t} \mathbbm{1}_{\{\tau_0>t\}})  \leq \left(1 - \frac{1-\delta}{N}\right)^t\leq e^{- \frac{t}{N}\left(1- \delta\right)}
		\end{align*}
		Putting $t = cN \log N$, we get:
		\begin{align*}
			\P(\tau_0 > t) \leq \P\left(e_t \mathbbm{1}_{\{\tau_0>t\}}> \frac{1}{N}\right) \leq N\E\left(e_t \mathbbm{1}_{\{\tau_0>t\}}\right)  \leq \frac{1}{N^{c\left(1- \delta\right)-1}}
		\end{align*}
		Since $1-\delta >0$, we have $\lim_{c\rightarrow \infty}\P(\tau_0 > cN\log N) = 0$, uniformly in $N.$ Then, we can apply the same arguments in the proof of Theorem 4.1 in \cite{levin2008} to complete the proof of the upper bound part of Theorem \ref{Theorem 1}.}

	For proving the lower bound, we use Theorem 3.1 (1) in \cite{Mukherjee_2021} to conclude that $$\bar{X}-c^* = O_{\P_{\beta,h,p}}(N^{-1/2}).$$
	If $L_t$ denotes the number of indices not selected by the Glauber dynamics by time $t$, then it follows from Lemma 7.12 in \cite{LevinPeresWilmer2006}, that:
	$$\E (L_t) = N\left(1-\frac{1}{N}\right)^t\quad\quad\text{and}\quad\quad \mathrm{Var} (L_t) \le \E (L_t)~.$$ This implies that at $t_N:=N\log N/4$, we have $\E(L_{t_N}) \sim N^{3/4}$, and hence,
	\begin{eqnarray*}
		\P\left(L_{t_N} \le \alpha \E L_{t_N}\right) &\le& \P\left(|L_{t_N}-\E L_{t_N}| \ge (1-\alpha) \E L_{t_N} \right)\\ &\le& \frac{\mathrm{Var} (L_{t_N})}{(1-\alpha)^2 \E^2 (L_{t_N})}\\ &\le& \frac{1}{(1-\alpha)^2 \E(L_{t_N})} = O_\alpha(N^{-3/4})
	\end{eqnarray*}
	for every $\alpha \in (0,1)$. 
	
	Now, consider the coupled chain $(\bm X^t,\bm Y^t)$ as constructed before, but this time, $\bm X^0$ is the vector of all $+1$s and $\bm Y^0$ is the vector of all $-1$s. If $\eta$ denotes the count function of the number of $+1$s in a configuration, then $\eta(\bm X^t)-\eta(\bm Y^t) \ge L_t$ and hence,
	$$\P\left(\eta(\bm X^{t_N}) - \eta(\bm Y^{t_N}) \le \alpha N^{3/4}\right) = O_\alpha(N^{-3/4})~.$$
	Noting that $\eta(\bm x) = N(1+\bar{x})/2$ for $\bm x \in \{-1,1\}^N$, we thus conclude that:
	$$\P\left(\bar{X}^{t_N} - \bar{Y}^{t_N} \le 2\alpha N^{-1/4}\right) = O_\alpha(N^{-3/4})~.$$ 
	We thus have $\P(\bar{X}^{t_N} - \bar{Y}^{t_N} \le 2\alpha N^{-1/4}) \rightarrow 0$, and hence by triangle inequality, the intersection of the events $E_{N} := \{|\bar{X}^{t_N} - c^*| \le \alpha N^{-1/4}\}$ and $F_{N} :=\{|\bar{Y}^{t_N} - c^*| \le \alpha N^{-1/4}\}$ has asymptotic probability $0$, which implies that either $E_N$ or $F_N$ (without loss of generality, let it be $E_N$) has probability at most $\frac{1}{2} + o(1)$. Hence, we have the following at $t_N = N\log N/4$:
	
	$$\max_{\bm x\in \mathcal{C}_N} d_{TV}\left(\P_{\bm x}(\bm X^{t_N} \in \cdot), \P_{\beta,h,p}\right) \ge \P_{\beta,h,p}(|\bar{X} - c^*|\le \alpha N^{-1/4}) - \P(E_N) \ge 1-\frac{1}{2} + o(1) = \frac{1}{2} + o(1)$$
	which completes the proof of the lower bound.

	\subsection{Proof of Theorem \ref{Theorem 1} for $p$-locally critical points $(\beta, h)$}
	To begin with, for each $k\in \{-N,-N+1,\ldots,N-1,N\}$, define $$S_k := \left\{\bm x\in \{-1,1\}^N: \sum_{i=1}^N x_i = k\right\}.$$
	Note that $S_{k}=\emptyset$ if $k+N$ is odd. Let \( m_1 \) denote the smallest local maximizer and \( m_2 \) the largest local maximizer of $H$. Since $H'(-1) =\infty$ and $H'(1)=-\infty$, there exists \( \varepsilon > 0 \) such that \( m_1 + \varepsilon < m_2 - \varepsilon \), \( H(m) < H(m_1) \) for all $m \in [-1,m_1+\varepsilon]\setminus \{m_1\}$, and \( H(m) < H(m_2) \) for all $m \in [m_2-\varepsilon,1]\setminus\{m_2\}$. Define:
	\[
	A := \bigcup_{k=-N}^{\lfloor (m_1+\varepsilon) N\rfloor} S_{k} \quad \quad \quad \text{and}\quad \quad \quad B:= \bigcup_{k=\lfloor (m_2-\varepsilon) N\rfloor}^{N} S_{k}~.
	\]
	Choose $N$ large enough such that $\lfloor (m_1+\varepsilon) N\rfloor < \lfloor (m_2-\varepsilon) N\rfloor$, whence $A\cap B =\emptyset$. Hence, one of $\P(A)$ and $\P(B)$ must be bounded by $1/2$, where $\P := \P_{\beta,h,p}$. Without loss of generality, assume that $\P(A) \le 1/2$. Let $P(\cdot,\cdot)$ denote the transition matrix of the Glauber dynamics $\bm X^t$. Define:
	$$Q(\bm x, \bm y) = \P(\bm x) P(\bm x,\bm y)\quad\quad\text{and}\quad \quad Q(A,B) := \sum_{\bm x \in A, \bm y \in B} Q(\bm x, \bm y)~.$$ Then, we have:
	$$\frac{Q(A,A^c)}{\P(A)} = \frac{1}{\P(A)}\sum_{k=-N}^{\lfloor (m_1+\epsilon)N\rfloor  } \sum_{\bm x \in S_k} \P(\bm x) \sum_{j = \lfloor (m_1+\epsilon)N\rfloor + 1}^N \sum_{\bm y \in S_j} P(\bm x, \bm y)$$
	The only non-zero transition probabilities in the above expression are contributed by transitions from $S_{\lfloor (m_1+\varepsilon)N\rfloor_{\#}}$ to $S_{(\lfloor (m_1+\varepsilon)N\rfloor+1)^{\#}}$, where for an integer $k$, we define $k_{\#}$ and $k^{\#}$ to be the integers $i$ nearest to $k$, not exceeding and not below $k$, respectively, such that $i+N$ is even. Without loss of generality, let us assume that $\lfloor (m_1+\varepsilon)N\rfloor + N$ is even. Then, 
	$$\frac{Q(A,A^c)}{\P(A)} = \frac{1}{\P(A)} \sum_{\bm x \in S_{\lfloor (m_1+\varepsilon)N\rfloor}} \P(\bm x)\sum_{\bm y \in S_{\lfloor (m_1+\varepsilon)N\rfloor+2}} P(\bm x, \bm y) \le \frac{\P(S_{\lfloor (m_1+\varepsilon)N\rfloor})}{\P(A)}~.$$
	Let $\bm X$ be a random spin configuration drawn from the measure $\P$, and let $\bar{X}$ denote its mean. Then, we have the following for all large $N$:
	$$\frac{\P(S_{\lfloor (m_1+\varepsilon)N\rfloor})}{\P(A)} \le \P\left(\bar{X} > m_1+\frac{\varepsilon}{2}\Big| \bar{X}  \le \frac{\lfloor (m_1+\epsilon)N\rfloor}{N}\right)~.$$
	The right side of the above inequality, in view of the proof of Lemma 3.7 in \cite{Mukherjee_2021}, can be bounded by:
	$$e^{N\left(H\left(m_1+\frac{\varepsilon}{2}\right)- H(m_1)\right)} O(N^{3/2}) = e^{-\Omega(N)}$$
	for all $\varepsilon > 0$ sufficiently small. Combining everything, we have:
	$$\frac{Q(A,A^c)}{\P(A)} \le e^{-\Omega(N)}$$
	Thus, the bottleneck ratio of the Glauber dynamics (see  Section 7.2 of \cite{LevinPeresWilmer2006} for details) is bounded by:
	$$\Phi_* := \min_{A:~ \P(A)\le 1/2} \frac{Q(A,A^c)}{\P(A)} \le e^{-\Omega(N)}~.$$
	By Theorem 7.3 in \cite{LevinPeresWilmer2006}, there exists a constant $C_\varepsilon>0$ such that:
	$$\tm (\varepsilon) \ge \frac{C_\varepsilon}{\Phi_*} \ge e^{\Omega_{\varepsilon}(N)}~,$$
	which completes the proof of Theorem \ref{Theorem 1} for locally critical points.

	\subsection{Proof of Theorem \ref{Theorem 1} for $p$-special points  $(\beta, h)$}
	By Lemma \ref{special_lemma}, the function $\lambda$ has a unique fixed point $c^*$ and $\lambda'(c^*) = 1$.  Define $e_t := c_t-c^*$, whence by Lemma \ref{lemma_1}, we have:
	\begin{align*}
		\E(e_{t+1}|e_t) = \left(1-\frac{1}{N}\right)e_t + \frac{1}{N}\left(\lambda(c_t)-c^*\right).
	\end{align*} 
	and if $e_t\ne 0$, then for all large $N$, we have:
	\[   
	\E(|e_{t+1}|~|e_t) = \left(1-\frac{1}{N}\right)|e_t| + \frac{1}{N} g(e_t) \sgn(e_t),
	\]
	where $g(u) := \lambda(c^*+u)-c^*$. 
	
	Define $\tau_0 = \text{min}\{t\geq 0: |e_t|\leq 1/N\}.$  Letting $\mathscr{A}_t := \sigma(e_0,\ldots,e_t)$, and noting that $g(0)=0$ and $\mathbbm{1}_{\tau_0 >t+1} \le \mathbbm{1}_{\tau_0>t}$, we have:
	$$\E (|e_{t+1}| \mathbbm{1}_{\tau_0 >t+1}~|\mathscr{A}_t) \le \left(1-\frac{1}{N}\right)|e_t|\mathbbm{1}_{\tau_0>t} + \frac{1}{N} g(e_t\mathbbm{1}_{\tau_0>t}) \sgn(e_t)~.$$
	Note that $e_t$ cannot change sign before time $\tau_0$, since otherwise, there would exist $s<\tau_0-1$ such that $|c_s-c_{s+1}| = |e_s-e_{s+1}| >2/N$, a contradiction. Hence, we have:
	\begin{equation}\label{sgfl}
		\E (|e_{t+1}| \mathbbm{1}_{\tau_0 >t+1}~|\mathscr{A}_t) \le \left(1-\frac{1}{N}\right)|e_t|\mathbbm{1}_{\tau_0>t} + \frac{1}{N} g(e_t\mathbbm{1}_{\tau_0>t}) \sgn(e_0).
	\end{equation}
	Defining $\theta_t := \E (|e_t|\mathbbm{1}_{\tau_0>t})$ and taking expectation on both sides of \eqref{sgfl}, we thus have:
	\begin{equation}\label{sgfl1}
		\theta_{t+1} \le \left(1-\frac{1}{N}\right) \theta_t + \frac{1}{N} \sgn(e_0) \E g(e_t\mathbbm{1}_{\tau_0>t}).
	\end{equation}
	Now, define the function $\bar{g}$ as the greatest convex minorant of $g$ on the interval $[-1-c^*,0]$ and the least concave majorant of $g$ on the interval $[0,1-c^*]$. Then, \eqref{sgfl1} gives us:
	$$\theta_{t+1} \le \left(1-\frac{1}{N}\right) \theta_t + \frac{1}{N} \sgn(e_0)\E \gb (e_t\mathbbm{1}_{\tau_0>t}), \quad \text{i.e.}\quad \theta_{t+1}-\theta_t \le -\frac{1}{N}\left(\theta_t - \sgn(e_0)\E \gb (e_t\mathbbm{1}_{\tau_0>t})\right)~.$$
	By Jensen's inequality, $\E \gb (e_t\mathbbm{1}_{\tau_0>t}) \le \gb(\theta_t)$ if $e_0>0$, and $\E \gb (e_t\mathbbm{1}_{\tau_0>t}) \ge \gb(-\theta_t)$ if $e_0<0$. Hence, 
	\[
	\theta_{t+1}-\theta_t \le \begin{cases} 
		-\frac{1}{N}(\theta_t - \gb(\theta_t)) & \text{if } e_0>0,\\
		-\frac{1}{N}(\theta_t + \gb(-\theta_t)) & \text{if } e_0<0.
	\end{cases}
	\]
	Since $0$ is the only fixed point of $g$ on the interval $[-1-c^*,1-c^*]$, the function $h(u) := g(u)-u$ cannot change sign on either side of $0$. By Lemma \ref{special_lemma}, $h'(0) = h''(0)=0$ and $h'''(0) < 0$, which implies the exitence of $\delta>0$, such that the following are true:
	\begin{enumerate}
		\item $h''<0$ on $(0,\delta)$ and $h''>0$ on $(-\delta,0)$;
		\item $h'<0$ on $(-\delta,\delta)\setminus \{0\}$;
		\item $h<0$ on $(0,\delta)$ and $h>0$ on $(-\delta,0)$.
	\end{enumerate}
	Point 3 above implies that $g(u)<u$ for all $u>0$ and $g(u)>u$ for all $u<0$, which is enough to conclude that $\gb(u) <u$ for all $u>0$ and $\gb(u)>u$ for all $u<0$. Now, suppose that $\theta_t >\mu$ for some constant $\mu>0$. Note that: 
	$$K_1(\mu) := \inf_{u \in [\mu,1-c^*]} (u-\gb(u)) > 0\quad\quad\text{and}\quad \quad K_2(\mu) := \inf_{u \in [-1-c^*,-\mu]} (\gb(u)-u) > 0.$$ 
	Taking $K(\mu) := \min \{K_1(\mu),K_2(\mu)\}$, we thus have:
	\begin{equation}\label{thetadrift7}
		\theta_{t+1}-\theta_t \le -\frac{K(\mu)}{N}~.
	\end{equation}
	It follows from \eqref{thetadrift7} that for every $\epsilon > 0$, there exists a time $T(\epsilon) = O_{\epsilon}(N)$ such that $\theta_t < \epsilon$ for all $t \ge T(\epsilon)$. Now, note that by point 1 above, $g$ is concave on a right neighborhood of $0$ and convex on a left neighborhood. Hence, there exists $\alpha>0$ such that $g=\gb$ on the interval $[-\alpha,\alpha]$, and $\lambda'''(u) <0$ for all $u\in [c^*-\alpha,c^*+\alpha]$. We wait till time $O(N)$ for $\theta_t$ to enter this interval, after which we have:
	$$\gb(\theta_t) = g(\theta_t) = \lambda^{'}(c^*)\theta_t + \frac{\lambda^{''}(c^*)}{2} \theta_t^2 + \frac{\lambda^{'''}(\xi_1)}{6}\theta_t^3\leq \theta_t - c_1 \theta_t^3$$
	and 
	$$\gb(-\theta_t) = g(-\theta_t) = - \lambda^{'}(c^*)\theta_t + \frac{\lambda^{''}(c^*)}{2}\theta_t^2 - \frac{\lambda^{'''}(\xi_2)}{6}\theta_t^3 \geq -\theta_t + c_2\theta_t^3 ,$$
	for some $\xi_1 \in [c^*, c^* + \theta_t]$ and $\xi_2 \in [c^* - \theta_t, c^*]$ and constants $c_1,c_2>0$. Hence, we have:
	\begin{equation}\label{prref}
		\theta_{t+1} \le \theta_t - \frac{c}{N}\theta_t^3
	\end{equation}
	for all $t\ge T(\alpha) = O_\alpha(N)$. Proceeding exactly similarly as Step 1 in the proof of Theorem 4.1 in \cite{levin2008}, we can now conclude that:
	$$\lim_{a\xrightarrow{}\infty}\P(\tau_0>aN^{3/2}) = 0$$
	uniformly in $N$. Further, following the similar approach as Step 2 in the same proof, we can conclude that $\tm = O(N^{3/2})$.
	
	For proving the lower bound, we will follow the approach in \cite{levin2008}. Towards this, define $\bar{e} := \bar{X}-c^*$ where $\bm X$ is simulated from the model \eqref{moddef}. It follows from Theorem 3.1 (3) in \cite{Mukherjee_2021} that under $\P_{\beta,h,p}$, $N^{1/4}\bar{e}$ converges weakly to a non-trivial distribution as $N\rightarrow \infty$. Fix a number $K \in (0,1)$ (to be chosen later), whence there exists $A>0$ (depending on $K$), such that for all large $N$,
	\begin{equation}\label{tv81}
		\P_{\beta,h,p}(|\bar{e}| \le A N^{-1/4}) \ge K~.
	\end{equation}
	Next, set $e_0 := 2AN^{-1/4}$ and define a chain $\tet$ with the same transition structure as $\et$, except at $e_0$, where the $\tet$ chain stays with probability equal to that of the $\et$ chain either going up or staying at $e_0$. The chains $\et$ and $\tet$ can be coupled in such a way, that $\et$ dominates $\tet$ stochastically under $\P_{e_0}$. Next, define:
	$$\tau := \min \{t\ge 0: \tet \le AN^{-1/4}\}$$ and let $d_t := e_0 - \tilde{e}_{t\wedge \tau}$.  Now, note that for $x\in (AN^{-1/4},e_0)$, we have:
	$$\E_{e_0}(\tett |\tet = x) = \E_{e_0}(\ett |\et = x) =\frac{1}{N}\left(\lambda(c^*+x)-c^*-x\right)+x~.$$
	A Taylor series expansion of the last term of the above identity, together with Lemma \ref{special_lemma} gives:
	\begin{equation*}
		\E_{e_0}(\tett |\tet = x) = x +\frac{x^3}{6N}\lambda'''(\xi) \ge  x - \frac{\alpha}{N} x^3
	\end{equation*}
	for some constant $\alpha >0$, where $\xi \in (c^*, c^*+x)$. Defining $\mathscr{A}_t:= \sigma(d_1,\ldots,d_t)$, we can now proceed exactly as in the proof of Theorem 4.2 in \cite{levin2008} to conclude that:
	\begin{equation}\label{lcbic}
		\E_{e_0}(d_{t+1}^2-d_t^2 |\mathscr{A}_t)  \le \frac{C_A}{N^2}
	\end{equation}
	for some constant $C_A>0$ depending on $A$. Taking expectation on both sides of \eqref{lcbic} and iterating, we have:
	$$\E_{e_0} d_t^2 \le \frac{C_A t}{N^2}~.$$
	Once again, following the steps in \cite{levin2008}, we have:
	$$\P_{e_0}(\tau \le t) \le \frac{C_A t}{A^2N^{3/2}}~.$$ Fixing another number $L\in (0,K)$ and setting $t := (A^2L/C_A) N^{3/2}$ above, we get:
	\begin{equation}\label{tv8278}
		\P_{e_0}(|e_t| \le A N^{-1/4}) \le \P_{e_0}(e_t \le A N^{-1/4}) \le \P_{e_0}(\tet \le A N^{-1/4}) \le \P_{e_0}(\tau \le t) \le L~.   
	\end{equation}
	Combining \eqref{tv81} and \eqref{tv8278}, we have: 
	$$\max_{\bm x\in \mathcal{C}_N} d_{TV}\left(\P_{\bm x}(\bm X^t \in \cdot), \P_{\beta,h,p}\right) \ge \P_{\beta,h,p}(|\bar{e}| \le A N^{-1/4}) - \P_{e_0}(|e_t| \le A N^{-1/4}) \ge K-L$$
	with $t = C_{K,L} N^{3/2}$ for some constant $C_{K,L}>0$.
	We can choose $K$ and $L$ such that $0<L<K<1$ and $K-L>\epsilon$ to conclude that $\tm(\epsilon) \ge C_{K,L} N^{3/2}$, which completes the proof of the lower bound.
	
	\section{Proof of Theorem \ref{Theorem 1mt}}\label{proofmrt}
	Throughout this section, let $S_t^+$ denote the sum of the $t^{\mathrm{th}}$ step of the restricted Glauber dynamics.   The upper bound result for the mixing time of the restricted dynamics is based on the following lemma, which relies on methods outlined in \cite{levin2008}.

	\begin{lemma} \label{prop:hit_from_zero}
		For every $c > 0$, define
	$\tau_\star = \tau_\star(c) := \min\{t \geq 0 : S^+_t \geq N m_+ + c N^{1/2} \}.$
Then,
		\begin{equation*} 
			\mathbb{E}_{\lceil Nm_-\rceil}[\tau_\star] = O(N \log N)
		\end{equation*}
	where the subscript $\lceil Nm_-\rceil$ denotes starting the restricted dynamics from $\lceil Nm_-\rceil$.
	\end{lemma}
	
	The proof of Lemma \eqref{prop:hit_from_zero} is similar to that of Proposition 5.2 in \cite{levin2008}, but for completeness, we work out the exact details in Appendix \ref{proofprophit}. Once again, following \cite{levin2008}, we prove the following lemma:
	
	\begin{lemma}\label{lem:alphw2}
		For every $\alpha>0$, define $\tau^* = \tau^*(\alpha) := \min\{t \geq 0 : S^+_t \leq N m_+ + \alpha N^{1/2} \}.$ Then, there exists a constant $c>0$ such that
		\begin{equation*}
			\lim_{N\rightarrow \infty} \P(\tau^*>cN\log N)=0.    
		\end{equation*}
	\end{lemma}
	
	\begin{proof}
		Suppose that the chain $S_t^+$ starts from a state larger than $Nm_+ +\alpha N^{1/2}$. It follows from the arguments in Lemmas \ref{thm1} and \ref{corollary_reg} that for any $\varepsilon>0$, with probability at least $1-e^{-\Omega(\sqrt{N})}$, the chain $s_t^+ := S_t^+/N$ enters and remains below $m_+ + \varepsilon$ for a duration of $e^{\Omega(\sqrt{N})}$ from time $O(N)$. Define $e_t := s^+_t - m_+$, and suppose that $\varepsilon > 0$ be such that $\delta := \sup_{x\in [m_+-\varepsilon, m_+ + \varepsilon]}\lambda^{'}(x)<1$. Repeating the proof of Section \ref{sec:regsy}, we once again have:
			\begin{align*}
			\E(e_{t} \mathbbm{1}_{\{\tau_0>t\}})  \leq \left(1 - \frac{1-\delta}{N}\right)^t\leq e^{- \frac{t}{N}\left(1- \delta\right)}~,
		\end{align*}
	where $\tau_0 := \min\{t\geq 0: e_t\leq \alpha N^{-1/2}\} = \tau^*$.
		Putting $t = cN \log N$, we thus get:
		\begin{align*}
			\P(\tau_0 > t) \leq \P\left(e_t \mathbbm{1}_{\{\tau_0>t\}}> \frac{\alpha}{\sqrt{N}}\right) \leq \alpha^{-1}\sqrt{N}\E\left(e_t \mathbbm{1}_{\{\tau_0>t\}}\right)  \leq \frac{N^{-c\left(1- \delta\right)+\frac{1}{2}}}{\alpha}
		\end{align*}
		Since $1-\delta >0$, we have $\lim_{c\rightarrow \infty}\P(\tau_0 > cN\log N) = 0$, uniformly in $N.$
	\end{proof}
	It now follows from Lemmas \ref{prop:hit_from_zero} and \ref{lem:alphw2}, and the proof of Theorem 5.3 in \cite{levin2008}, that:
	\begin{equation}\label{ubwl7}
	\tau_{\textbf{mix}} = O(N\log N).
\end{equation}
	
	On the other hand, the following central limit theorem for the mean magnetization under the stationary measure $\mu_+ := \P_{\beta, h, p}(\cdot\Big| \bar{X} \ge m_-)$ (see \cite{Mukherjee_2021}):
	$$\bar{X}-m_+ = O_{\mu_+} (N^{-1/2})$$
	together with the exact arguments presented in the proof of Theorem 5.4 in \cite{levin2008}, also establishes a matching lower bound $\Omega(N\log N)$ to $\tau_{\mathbf{mix}}$. This, coupled with \eqref{ubwl7}, completes the proof of Theorem \ref{Theorem 1mt}.
	\section{Discussion}
	In this paper, we determined the mixing time order of the Glauber dynamics for the $p$-spin Curie-Weiss model on different regions of almost the entire parameter space $\{(\beta,h):\beta > 0, h \in \mathbb{R}\}$. The asymptotics of the magnetization in the $p$-spin Curie-Weiss model is largely governed by the global maximizers of a function $H_{\beta,h,p}$, while the mixing time of the corresponding Glauber dynamics depends on the number and nature of the local maximizers of this function. We showed that the mixing time is $\Theta(N\log N)$ whenever $H_{\beta,h,p}$ has a unique stationary point, which is a local maximizer with negative curvature, it is $\exp(\Omega(N))$ when there are multiple local maximizers, and it is $\Theta(N^{3/2})$ when there is a unique local maximizer with zero curvature. 
	
	The above three regions almost exhaust the entire parameter space, with just a one-dimensional curve left out that separates the region of multiple local maximizers from the region of unique local maximizer with negative curvature. Everywhere on this curve, there exists a stationary inflection point in addition to a local maximizer, and the mixing time of the Glauber dynamics is likely governed by the behavior of the mean-magnetization chain near this point of inflection. This is left as an open question. Another interesting direction for future research might be to analyze the Glauber dynamics mixing time for the $p$-spin Curie-Weiss Potts model \citep{bhowal2023limittheoremsphasetransitions, bhowal2024ratesconvergencemagnetizationtensor}. This problem has been solved for the $2$-spin case in \cite{Cuff2012GlauberDF} (also see \cite{he1}). On a different note, it was shown in \cite{levin2008}, that for $\beta < \frac{1}{2}$, the Glauber dynamics for the $2$-spin Curie-Weiss model with no external field exhibits a cut-off at time $[2(1-2\beta)]^{-1} N\log N$ with window size $N$, i.e. the total variation distance of the dynamics from the stationary distribution drops from near $1$ to near $0$ in a window of order $N$ around $[2(1-2\beta)]^{-1} N\log N$. It would be interesting to investigate a similar cut-off phenomenon on the set $\mathfrak{R}_p$ for the Glauber dynamics on the $p$-spin Curie-Weiss model. Finally, as already mentioned before, \cite{Ding2014} gave a thorough description of the mixing time transition of the Curie-Weiss Glauber dynamics for the case $p=2, h=0$, from $\Theta(N\log N)$ for $\beta <\frac{1}{2}$ through $\Theta (N^{3/2})$ at $\beta=\frac{1}{2}$ to $\exp(\Theta(N))$ for $\beta > \frac{1}{2}$. The relevant question here, is that can one derive an analogous description of the transition of the mixing time by zooming into the special point(s) and possibly along some suitable curve passing from $\mathfrak{R}_p$ to $\mathfrak{C}_p$ through these special points?
	
	We came to know of the recent paper \citep{mikulincer2024stochasticlocalizationnongaussiantilts} that was arxived around the same time this manuscript was finalized, where the authors give some mixing time estimates for the Glauber dynamics on tensor Curie-Weiss models with no external field, as an application of a more general framework extending Eldan's stochastic localization process to include non-Gaussian tilts. In particular, our mixing threshold $\beta_p'$ (see Appendix \ref{appendix_geo}) appears in their Proposition 5.1 also, where they show exponentially slow mixing for the Glauber dynamics on the tensor Curie-Weiss model with no external field, above this threshold. However, to the best of our understanding, explicit results on the mixing time of Glauber dynamics on the $p$-spin Curie-Weiss model for the cases $h\ne 0$ and $h=0, \beta \le \beta_p'$ are not stated in this paper.

	\section{Acknowledgement}
	Somabha Mukherjee was supported by the National University of Singapore
	Start-Up Grant R-155-000-233-133 and the AcRF Tier 1 grant A-8001449-00-00. Ramkrishna Jyoti Samanta acknowledges support from the Additional Funding Programme for Mathematical Sciences, delivered by
	EPSRC (EP/V521917/1) and the Heilbronn Institute for Mathematical Research. R.J.S. would also like to thank the Isaac Newton Institute for Mathematical Sciences, Cambridge, for support during the programme Stochastic systems for anomalous diffusion, where this work was undertaken. He was supported by EPSRC (EP/R014604/1) through finance division of the University of Cambridge. 
	
	S.M. is grateful to Bhaswar Bhattacharya for several useful discussions throughout the preparation of this manuscript, and would like to thank Nathan Ross for pointing out related literature on fast mixing of the restricted Glauber dynamics in the Curie-Weiss Potts model in presence of metastable states.

		\bibliography{references}
		
		\appendix
		
		\section{Geometry of the Mixing Phases}\label{appendix_geo}
		
		In this section, we provide a detailed description of the geometry of the three regions $\mathfrak{R}_p, \mathfrak{C}_p$ and $\mathfrak{S}_p$. To begin with, note that in view of Lemma F.2 in \cite{Mukherjee_2021}, 
		(and as already mentioned in the proof of Lemma \ref{special_lemma}), we know that for odd $p$, $\mathfrak{S}_p = \{(\hat{\beta}_p, \hat{h}_p)\}$ and for even $p$, $\mathfrak{S}_p = \{(\hat{\beta}_p, \hat{h}_p), (\hat{\beta}_p, -\hat{h}_p)\}$, where
		$$\hat{\beta}_p = \frac{1}{2(p-1)}\left(\frac{p}{p-2}\right)^{\frac{p-2}{2}},
		\quad \hat{h}_p = \tanh^{-1}\left(\sqrt{\frac{p-2}{p}}\right) - \hat{\beta}_p p \left(\frac{p-2}{p}\right)^{\frac{p-1}{2}}.$$ Also, define the following two threshold parameters, which are important quantities in describing the phase geometry:
		$$\tilde{\beta}_p := \sup \left\{\beta>0: \sup_{x\in [0,1]} H_{\beta,0,p}(x) =0\right\} =\inf_{x\in [0,1]} \frac{I(x)}{x^p}~,$$ $$\beta_p' := \sup \left\{\beta>0: \sup_{x\in [0,1]} H_{\beta,0,p}'(x) =0\right\} = \inf_{x\in [0,1]} \frac{\tanh^{-1}(x)}{px^{p-1}}.$$
		
		We now proceed to describe the set $\mathfrak{C}_p$.
		
		\begin{lemma}\label{pcritdescr}
			The geometry of $\mathfrak{C}_p$ depends on the parity of $p$. In particular, we have:
			\begin{enumerate}
				\item[(1)] If $p$ is odd, then there exist strictly decreasing, continuous functions $U,C,L : (\hat{\beta}_p,\infty) \to \mathbb{R}$, such that:
				\begin{enumerate}
					\item $C(\tbp)=0$, $C(\beta) \rightarrow -\infty$ as $\beta \rightarrow \infty$, and the set of all points $(\beta,h)$ such that $H_{\beta,h}$ has more than one global maximizer, is given by $\{(\beta,C(\beta)): \beta > \hat{\beta}_p\}$;
					\item $L(\beta)< C(\beta)<U(\beta)$ for all $\beta > \hat{\beta}_p$;
					\item $\mathfrak{C}_p = \{(\beta,h): \beta >\hat{\beta}_p, ~h\in (L(\beta),U(\beta))\}$.
					\item $U(\beta) >0$ for all $\beta >\hat{\beta}_p$ and $\lim_{\beta\rightarrow\infty} U(\beta)=0$;
					\item As $\beta \rightarrow \hat{\beta}_p$ from the right, all three functions $U(\beta), C(\beta)$ and $L(\beta)$ converge to $\hat{h}_p$.        
				\end{enumerate}
				\item[(2)] If $p$ is even, then $\beta_p' \in (\hat{\beta}_p,\tbp)$, and there exist continuous functions $U,C: (\hat{\beta}_p,\infty) \to \mathbb{R}$ and a strictly decreasing, continuous function $L: (\hat{\beta}_p,\beta_p']\to \mathbb{R}$, such that:
				\begin{enumerate}
					\item $C$ is strictly decreasing on $(\hat{\beta}_p,\tilde{\beta}_p)$, vanishing on $[\tbp,\infty)$, and the set of all points $(\beta,h)$ such that $H_{\beta,h}$ has more than one global maximizer, is given by $\{(\beta,\pm C(\beta)): \beta > \hat{\beta}_p\}$;
					\item $0<L(\beta)< U(\beta)$ for all $\beta \in (\hat{\beta}_p,\beta_p')$, $L(\beta_p')=0$, and the graphs of $\pm C$ lie in the interior of the region bounded by the graphs of $\pm U$ and $\pm L$. 
					
					\item The set $\mathfrak{C}_p$ is given by:
					\begin{eqnarray*}
						\mathfrak{C}_p &:=& \{(\beta,h): \beta >\beta_p', ~h\in (-U(\beta),U(\beta))\}\\ &\bigcup& \{(\beta,\pm h): \beta \in (\hat{\beta}_p,\beta_p'], ~h\in (L(\beta),U(\beta))\}
					\end{eqnarray*}
					\item There exists $\beta_0>\beta_p'$, such that $U$ is strictly decreasing on $(\hat{\beta}_p,\beta_0)$ and strictly increasing above $\beta_0$. Further, $\lim_{\beta \rightarrow \infty} U(\beta) =\infty$.
					\item As $\beta \rightarrow \hat{\beta}_p$ from the right, all three functions $U(\beta), C(\beta)$ and $L(\beta)$ converge to $\hat{h}_p$.        
				\end{enumerate}
			\end{enumerate}
		\end{lemma}
		
		\begin{proof}[Proof of Part (1)]
			The existence of a strictly decreasing and continuous function $C : (\hat{\beta}_p,\infty) \to \mathbb{R}$ satisfying (a) follows immediately from Lemma F.3 in \cite{Mukherjee_2021}. Further, it follows from the proof of Lemma F.3 in \cite{Mukherjee_2021}, that if $\beta < \hat{\beta}_p$, then $H'' < 0$ on $[-1,1]$, i.e. $H$ is strictly concave on $[-1,1]$, and hence cannot have more than one local maximizer. Moreover, if $\beta = \hat{\beta}_p$, then by Lemma F.2 (1) in \cite{Mukherjee_2021}, $H''$ has a unique root in $[-1,1]$. This, together with the fact that  $H''(1)=H''(-1) = -\infty$, implies that $H''<0$ everywhere else on $[-1,1]$. Hence, once again $H$ is strictly concave, so cannot have more than one local maximizer. So, any point $(\beta,h) \in \mathfrak{C}_p$ must satisfy $\beta > \hat{\beta}_p$. 
			
			Next, it follows from the proof of Lemma F.3 (2) in \cite{Mukherjee_2021} that there exist $0< a_1 := a_1(\beta)< a_2 := a_2(\beta)<1$ such that $H_{\beta,0,p}'' < 0$ on $[-1,a_1)\bigcup (a_2,1]$, $H_{\beta,0,p}'' = 0$ on $\{a_1,a_2\}$, and $H_{\beta,0,p}'' > 0$ on $(a_1,a_2)$. This implies that for all $h$, $H_{\beta,h,p}'$ is strictly decreasing on $[-1,a_1)$, strictly increasing on $(a_1,a_2)$ and strictly decreasing on $(a_2,1)$. Define: 
			$$U(\beta) := \tanh^{-1}(a_1) -\beta p a_1^{p-1}\quad\quad \text{and} \quad \quad L(\beta) := \tanh^{-1}(a_2)-\beta p a_2^{p-1}~.$$ Note that $H_{\beta,h,p}'(1) = -\infty$ and $H_{\beta,h,p}'(-1) = \infty$. If $h>U(\beta)$, then $H_{\beta,h,p}'(a_1) > 0$, and hence, $H_{\beta,h,p}'$ must have exactly one root (in $(a_2,1)$), which is a local maximizer of $H_{\beta,h,p}$. Also, if $h<L(\beta)$, then $H_{\beta,h,p}'(a_2)< 0$, which implies that $H_{\beta,h,p}'$ must have exactly one root (in $(-1,a_1)$), that is a local maximizer of $H_{\beta,h,p}$. Hence, for $h>U(\beta)$ or $h<L(\beta)$, the point $(\beta,h)\notin \mathfrak{C}_p$. Moreover, if $h=U(\beta)$, then $a_1$ is a stationary inflection point of $H_{\beta,h,p}$, with exactly one more stationary point to the right of $a_2$ which is still a local maximum, and if $h=L(\beta)$, then $a_2$ is a stationary inflection point of $H_{\beta,h,p}$, with exactly one more stationary point to the left of $a_1$ which is still a local maximum. Hence, the points $(\beta,U(\beta))$ and $(\beta,L(\beta))$ are not in $\mathfrak{C}_p$ either. 
			
			Next, since $H_{\beta,0,p}'$ is strictly increasing on $(a_1,a_2)$, it follows that $U(\beta) > L(\beta)$. Further, we have precisely shown that a necessary condition for $(\beta,h)$ to be in $\mathfrak{C}_p$, is that $L(\beta) < h<U(\beta)$. This proves (b). Now, suppose that $\beta >\hat{\beta}_p$ and $L(\beta) < h<U(\beta)$. Then, $H_{\beta,h,p}'(a_1) <0$ and $H_{\beta,h,p}'(a_2) > 0$, which implies the existence of two roots of $H_{\beta,h,p}'$ in $(-1,a_1)$ and $(a_2,1)$, both of which are local maximizers of $H_{\beta,h,p}$.  This completes the proof of (c).
			
			Since $a_1(\beta)$ and $a_2(\beta)$ are roots of $H_{\beta,0,p}''$, they are continuous functions themselves. Hence, both $U$ and $L$ are continuous. Further, for a differentiable function $f: (\hat{\beta}_p,\infty) \to (0,1)$, if we define $V_f(\beta) := \beta p f(\beta)^{p-1} - \tanh^{-1}(f(\beta))$, then we have:
			$$V_f'(\beta) = pf(\beta)^{p-1} + f'(\beta)H_{\beta,0,p}''(f(\beta)).$$
			Since $U = -V_{a_1}$ and $L = - V_{a_2}$, we have
			$U'(\beta) = -pa_1(\beta)^{p-1} <0 $ and $L'(\beta) = -pa_2(\beta)^{p-1} < 0$. Hence, both $U$ and $L$ are strictly decreasing.
			
			Now, note that $H_{\beta,0,p}'(a_1) <0$, since $H_{\beta,0,p}'(0) = 0$ and $H_{\beta,0,p}'$ is strictly decreasing on $[0,a_1)$. So, $U(\beta) = -H_{\beta,0,p}'(a_1) > 0$. Next, suppose that $\beta$ is large enough, so that:
			$$y := \left(\frac{5}{3\beta p(p-1)}\right)^{\frac{1}{p-2}} <\frac{1}{2}~.$$ Then, we have:
			$$H_{\beta,0,p}''(y) = \frac{5}{3} - \frac{1}{1-y^2} > 0$$ and hence, $a_1(\beta)<y$. So, $a_1(\beta) \rightarrow 0$ as $\beta \rightarrow \infty$, and hence, $\tanh^{-1}(a_1(\beta)) \rightarrow 0$. Further,
			$$\beta p a_1(\beta)^{p-1} = O_p\left(\beta^{1-\frac{p-1}{p-2}}\right) = O_p\left(\beta^{-\frac{1}{p-2}}\right)~.$$ This proves (d). Finally, note that as $\beta \rightarrow \hat{\beta}_p^+$, both the roots $a_1(\beta)$ and $a_2(\beta)$ of $H_{\beta,0,p}''$ also converge to the only root $\sqrt{1-\frac{2}{p}}$ of $H_{\hat{\beta}_p,0,p}''$. Hence, both $U(\beta)$ and $L(\beta)$ converge to $\hat{h}_p$. This completes the proof of part (1) of Lemma \ref{pcritdescr}. 
		\end{proof}
		\begin{proof}[Proof of Part (2)]
			The existence of a continuous function $C : (\hat{\beta}_p,\infty) \to \mathbb{R}$ satisfying (a) follows immediately from Lemma F.3 in \cite{Mukherjee_2021}. Also, once, again, $H$ is strictly concave for $\beta <\hat{\beta}_p$. For $\beta = \hat{\beta}_p$, $\pm \sqrt{1-\frac{2}{p}}$ are the only roots of $H''$, which coupled with the facts that $H''(\pm 1) = -\infty$ and $H''(0) = -1$, implies that $H''<0$ everywhere else on $[-1,1]$, i.e. $H$ is once again strictly concave. So, $H$ cannot have multiple local maximizers for $\beta \le \hat{\beta}_p$, i.e. any point $(\beta,h) \in \mathfrak{C}_p$ must once again satisfy $\beta > \hat{\beta}_p$. 
			
			Once again, it follows from the proof of Lemma F.3 (1) in \cite{Mukherjee_2021} that there exist $0< a_1 := a_1(\beta)< a_2 := a_2(\beta)<1$ such that $H_{\beta,0,p}'' < 0$ on $[0,a_1)\bigcup (a_2,1]$, $H_{\beta,0,p}'' = 0$ on $\{a_1,a_2\}$, and $H_{\beta,0,p}'' > 0$ on $(a_1,a_2)$. Since $H''$ is an even function, it follows that $H_{\beta,0,p}'' <0$ on $[-1,-a_2)\bigcup (-a_1,0]$, $H_{\beta,0,p}'' = 0$ on $\{-a_1,-a_2\}$, and $H_{\beta,0,p}'' > 0$ on $(-a_2,-a_1)$. This implies that for all $h$, $H_{\beta,h,p}'$ is strictly decreasing on $[-1,-a_2)$, strictly increasing on $(-a_2,-a_1)$, strictly decreasing on $(-a_1,a_1)$, strictly increasing on $(a_1,a_2)$ and finally, strictly decreasing on $(a_2,1]$. Define:
			$$U(\beta) := -\min\{H_{\beta,0,p}'(-a_2), H_{\beta,0,p}'(a_1)\}\quad \quad \text{and}\quad\quad U^-(\beta) := -\max\{H_{\beta,0,p}'(-a_1), H_{\beta,0,p}'(a_2)\}.$$
			Since $H_{\beta,0,p}'$ is an odd function, we have $U^-(\beta) =-U(\beta)$. If $h> U(\beta)$, then both $H_{\beta,h,p}'(a_1)$ and $H_{\beta,h,p}'(-a_2) > 0$, and hence, $H_{\beta,h,p}'$ must have exactly one root (in $(a_2,1)$), which is a local maximizer of $H_{\beta,h,p}$. Similarly, if $h<U^-(\beta)$, then both $H_{\beta,h,p}'(-a_1)$ and $H_{\beta,h,p}'(a_2) < 0$, which implies that $H_{\beta,h,p}'$ must have exactly one root (in $(-1,-a_2)$), which is a local maximizer of $H_{\beta,h,p}$. Hence, for $h>U(\beta)$ or $h<U^-(\beta)$, the point $(\beta,h)\notin \mathfrak{C}_p$. Moreover, if $h=U(\beta)$, then either $a_1$ or $-a_2$ is a stationary inflection point of $H_{\beta,h,p}$, with exactly one more stationary point to the right of $a_2$ which is still a local maximum, and if $h=U^-(\beta)$, then either $a_2$ or $-a_1$ is a stationary inflection point of $H_{\beta,h,p}$, with exactly one more stationary point to the left of $-a_2$ which is still a local maximum. Hence, the points $(\beta,U(\beta))$ and $(\beta,U^-(\beta))$ are not in $\mathfrak{C}_p$ either. Thus, a necessary condition for $(\beta,h)$ to be in $\mathfrak{C}_p$, is that $-U(\beta) < h<U(\beta)$, which also shows that $C(\beta) < U(\beta)$. Also, continuity of $U$ follows from continuity of $a_1$ and $a_2$.
			
			Next, it follows from the proof of part (1), that the map $\beta \mapsto H_{\beta,0,p}'(a_2(\beta))$ is strictly increasing. It also follows from the monotonicity pattern of $H_{\beta,0,p}'$, that:
			$$\sup_{x\in [0,1]} H_{\beta,0,p}'(x) = \max\{0, H_{\beta,0,p}'(a_2)\}~.$$
			Hence, $H_{\beta,0,p}'(a_2) \le 0$ for $\beta \le \beta_p'$ and  $H_{\beta,0,p}'(a_2) > 0$ for $\beta > \beta_p'$. It is also easy to see that $\beta_p' \in (\hat{\beta}_p,\tbp)$, since $H_{\hat{\beta}_p,0,p}'$ is strictly negative on the positive side, and $H_{\tbp,0,p}'$ is positive somewhere on the positive side (because $H_{\tbp,0,p}<0$ on a right neighborhood of $0$, and must be $0$ somewhere on the positive side).
			
			At this point, we observe that for even $p$, $x$ is a local maximizer of $H_{\beta,h,p}$ if and only if $-x$ is a local maximizer of $H_{\beta,-h,p}$, which implies that the phase diagram is symmetric about the $h=0$ axis. Hence, it suffices to focus on the $h\ge 0$ half of the plane. In view of Lemma \ref{derlv} (2), we can now conclude that for $\beta > \beta_p'$, the condition $-U(\beta)<h<U(\beta)$ is also sufficient for $(\beta,h)$ to be in $\mathfrak{C}_p$. For $\beta \in (\hat{\beta}_p,\beta_p']$, if we define $L(\beta) :=-H_{\beta,0,p}'(a_2)$, then $L$ is continuous on its domain (since $a_2$ is continuous). Also, note that in this case, $U(\beta) = - H_{\beta,0,p}'(a_1) \ge L(\beta)$, since $H_{\beta,0,p}'$ is increasing on $(a_1,a_2)$. It now follows from Lemma \ref{derlv} (1), that a necessary and sufficient condition for $(\beta,h)$ to belong to $\mathfrak{C}_p$, is that $h \in \pm (L(\beta),U(\beta))$. This proves (c). 
			
			Next, since $H_{\beta,0,p}'(a_2) \le 0$ for $\beta \le \beta_p'$ and  $H_{\beta,0,p}'(a_2) > 0$ for $\beta > \beta_p'$, by continuity, we must have $H_{\beta_p',0,p}'(a_2(\beta_p')) = 0$ i.e. $L(\beta_p') = 0$. Since $L$ is strictly decreasing on its domain, the proof of (b) is now complete.
			
			To prove (d), note that if $H_{\beta,0,p}'(-a_2)\le H_{\beta,0,p}'(a_1)$, then $U(\beta) = H_{\beta,0,p}'(a_2)$, and otherwise, $U(\beta)= -H_{\beta,0,p}'(a_1)$. In the first case, $U'(\beta)>0$ and in the second case, $U'(\beta)<0$. Define $g: (\hat{\beta}_p ,\infty)\to \mathbb{R}$ as: $$g(\beta) := H_{\beta,0,p}'(a_1(\beta)) - H_{\beta,0,p}'(-a_2(\beta)).$$
			Note that $g$ is a strictly increasing and continuous function, and $g(\beta_p') = H_{\beta_p',0,p}'(a_1(\beta_p')) < 0$. Next, since $a_1 < w_p := \sqrt{1-\frac{2}{p}}$, we have $H_{\beta,0,p}'(a_1(\beta)) \ge - \tanh^{-1}(w_p)$. On the other hand, we can choose $\beta >0$ large enough, such that $H_{\beta,0,p}'(\frac{1}{2}) > \tanh^{-1}(w_p) > 0$, which would then imply that $H_{\beta,0,p}'(a_2(\beta)) > \tanh^{-1}(w_p)$. So, $g(\beta) > 0$ for all large $\beta$. Hence, there is a unique root $\beta_0$ of $g$, larger than $\beta_p'$, and for $\beta \in (\hat{\beta}_p,\beta_0)$, $U$ is strictly decreasing, while for $\beta \ge \beta_0$, $U$ is strictly increasing. Moreover, since $H_{\beta,0,p}'(\frac{1}{2}) \rightarrow \infty$ as $\beta \rightarrow \infty$, we have:
			$$\lim_{\beta \rightarrow \infty} U(\beta ) = H_{\beta,0,p}'(a_2(\beta)) = \infty~,$$
			which proves (d).
			
			Finally, once again, note that as $\beta \rightarrow \hat{\beta}_p^+$, both the positive roots $a_1(\beta)$ and $a_2(\beta)$ of $H_{\beta,0,p}''$ also converge to the only positive root $\sqrt{1-\frac{2}{p}}$ of $H_{\hat{\beta}_p,0,p}''$. Hence, both $U(\beta)$ and $L(\beta)$ converge to $\hat{h}_p$. Hence, $H_{\beta,0,p}'(a_1) \rightarrow -\hat{h}_p$ and $H_{\beta,0,p}'(-a_2) = -H_{\beta,0,p}'(a_2) \rightarrow \hat{h}_p$. This proves (e) and completes the proof of Lemma \ref{pcritdescr}.  
		\end{proof}
		
		Finally, we describe the set $\mathfrak{R}_p$. Towards this, let us recall the notations in Lemma \ref{pcritdescr}, and define the set $\mathfrak{B}_p$ to be the union of the graphs of the functions $U$ and $L$ when $p$ is odd, and union of the graphs of the functions $\pm U$ and $\pm L$ when $p$ is even.

		\begin{lemma}\label{regrgam}
			$\mathfrak{R}_p = (\overline{\mathfrak{C}_p})^c$, where $\overline{S}$ denotes the closure of a set $S\subseteq \Theta$ with respect to the Euclidean topology.
		\end{lemma}
		\begin{proof}
			It is clear from the proofs of Lemma \ref{pcritdescr} and Lemma \ref{derlv}, that
			$\partial \mathfrak{C}_p = \mathfrak{B}_p\bigcup \mathfrak{S}_p$ (here $\partial$ denotes the topological boundary), and that for all $(\beta,h) \in \mathfrak{B}_p$, the function $H_{\beta,h,p}$ has exactly one local maximizer with the remaining stationary point(s) being point(s) of inflection. This shows that $\mathfrak{R}_p \subseteq (\overline{\mathfrak{C}_p})^c$. On the other hand, it also follows from the proofs of Lemma \ref{pcritdescr} and Lemma \ref{derlv}, that for all $\beta >\hat{\beta}_p$, the set of all points $(\beta,h)$ in the complement of $\overline{\mathfrak{C}_p}$ (i.e. the set of all $(\beta,h)$ such that $h>U(\beta)$ or $h<L(\beta)$ for odd $p$, and $\pm h > U(\beta)$ or $\pm h <L(\beta)$ for even $p$) belongs to $\mathfrak{R}_p$. 
			
			Next, for $\beta <\hat{\beta}_p$, $H_{\beta,h,p}''<0$ on $[-1,1]$, so $H$ has exactly one stationary point, which is a global maximizer. Finally, at $\beta=\hat{\beta}_p$, $H_{\beta,h,p}''<0$ everywhere on $[-1,1]\setminus D$ where the set $D$ consists of $\sqrt{1-2/p}$ if $p$ is odd, and $\pm \sqrt{1-2/p}$ if $p$ is even. Hence, $H_{\beta,h,p}$ once again has exactly one local maximizer. Also, note that $H_{\beta,h,p}'(\sqrt{1-2/p}) = h-\hat{h}_p$ for all $p$, and $H_{\beta,h,p}'(-\sqrt{1-2/p}) = h+\hat{h}_p$ for even $p$. Hence, no point in $D$ can be a stationary point of $H_{\beta,h,p}$ for odd $p$ with $h\ne \hat{h}_p$, and for even $p$ with $h\ne \pm \hat{h}_p$. So, the unique local maximizer must lie outside $D$, where the second derivative is negative. This shows that $\mathfrak{R}_p \supseteq (\overline{\mathfrak{C}_p})^c$, completing the proof of Lemma \ref{regrgam}.
		\end{proof}

		In conclusion, the sets $\mathfrak{R}_p$, $\mathfrak{C}_p$, $\mathfrak{S}_p$ and $\mathfrak{B}_p$ form a partition of the parameter space $\Theta$. Note that for every point $(\beta,h) \in \mathfrak{B}_p$, the function $H_{\beta,h,p}$ has exactly one local maximizer with the remaining stationary point(s) being point(s) of inflection.

		\section{Proof of Lemma \ref{prop:hit_from_zero}} \label{proofprophit}
		To begin with, note that $U_t := S_t^+/2$ is a birth-and-death chain on the state space $$\tilde{R} := \left\{\frac{N}{2}, \frac{N}{2}-1,\ldots,\frac{N}{2}-v\right\}$$ where $v$ is the largest integer such that $N-2v \ge \lceil Nm_-\rceil$. If $N$ is odd, the chain $U_t$ is not integer valued, and then one can just shift all states by $-1/2$.  By Equation (5.7) in \cite{levin2008}, we know that
		$$\mathbb{E}_{\ell-1} (\tau_\ell) \le \frac{1}{q_\ell \pi^{(\ell)}(\ell)}$$
		for all $\frac{N}{2}-v < \ell \le \frac{\lceil N m_+ + c N^{1/2}\rceil}{2}$, where $\pi^{(\ell)}$ denotes the stationary distribution of the chain $U_t$ restricted on the set $\{\frac{N}{2}-v, \frac{N}{2}-v+1\ldots, \ell\}$, $q_\ell$ denotes the transition probability of $U_t$ going down from $\ell$, and 
		$$\tau_\ell := \inf\{t\ge 0: U_t = \ell\}~.$$
		Since, $q_\ell$ is bounded away from $0$ uniformly over $\ell > \frac{N}{2}-v$, we have:
		\begin{equation}\label{elm1tl}
			\mathbb{E}_{\ell-1} (\tau_\ell) \le C \frac{ \sum_{j=\tilde{\ell}}^{\ell} \binom{N}{N(1+y)/2} \exp\left( N\beta y^p + Nh y \right)}{
			\binom{N}{N(1+x)/2} \exp\left( N\beta x^p + Nh x \right)}
		\end{equation}
		for some constant $C>0$, uniformly over all $\frac{N}{2}-v < \ell \le \frac{\lceil N m_+ + c N^{1/2}\rceil}{2}$, $\ell = Nx/2$, $j = Ny/2$ and $\tilde{\ell}:= \frac{N}{2}-v$. Applying Stirling's approximation to \eqref{elm1tl}, we have the following for all large $N$ and $\ell = \tilde{\ell} + O(\sqrt{N\log N})$:
		
		\begin{equation}\label{firststbn7}
				\mathbb{E}_{\ell-1} (\tau_\ell) \le 2C \sum_{j=\tilde{\ell}}^{\ell} \exp\left[N(H_{\beta,h,p}(y)-H_{\beta,h,p}(x))\right] \le  2C(\ell-\tilde{\ell}) = O(\sqrt{N\log N}).  
		\end{equation}
		 Note that the second inequality in \eqref{firststbn7} is true because $H$ is increasing in the interval $[m_-, m_+]$ within which both $x$ and $y$ lie. 
		 
		 Next, suppose that  $\tilde{\ell} + K \sqrt{N\log N} \le \ell \le \ell_0$, where $\ell_0 := (i_0-\delta) N/2$ with $i_0$ being the unique infletion point of $H$ between $m_-$ and $m_+$, and $K, \delta>0$ are constants to be chosen suitably later.  First, note that if $j \le (\tilde{\ell}+\ell)/2$, then

		 \begin{eqnarray*}
			 \exp\left[-N(H(x)-H(y))\right]
			&=&  \exp\left[ -N(x-y)H'(y) - N(x-y)^2H''(\xi)/2 \right]\nonumber\\&\le&  \exp\left[- \frac{(\ell-\tilde{\ell})^2H''(\xi)}{2N} \right]\nonumber
		\end{eqnarray*}
		for some $\xi \in [y,x]$ between $x$ and $y$, where the last inequality followed from the fact that both $H'$ and $H''$ are non-negative on $[y,x]$. In fact, since $x\le i_0-\delta$, we can lower bound $H''(\xi)$ by a positive constant $\gamma$ (depending on $H$ and $\delta$). Combining all these, we have:
		
		\begin{equation*}
			\exp\left[-N(H(x)-H(y))\right]\le  \exp\left[-\frac{\gamma}{2N}(K^2 N\log N)\right] =  N^{-\frac{\gamma K^2}{2}}.
		\end{equation*}
	We may now choose $\delta$ arbitrarily small (but fixed), and $K \ge \sqrt{2/\gamma}$, so that:
	\begin{equation}\label{secbound46}
		\exp\left[-N(H(x)-H(y))\right] = O\left(\frac{1}{N}\right)
	\end{equation}
	whenever $j \le (\tilde{\ell}+\ell)/2$. Now suppose that $j > (\tilde{\ell}+\ell)/2$. Note that:
	
	$$\exp\left[-N(H(x)-H(y))\right] = \exp\left[ -N(x-y)H'(\xi_{j,\ell})\right] = \exp\left[ -2(\ell-j)H'(\xi_{j,\ell})\right]$$
	for some $\xi_{j,\ell} \in [y,x]$. Also, there exists $u_{j,\ell} \in [m_-,\xi_{j,\ell}]$ such that:
	$$H'(\xi_{j,\ell}) = (\xi_{j,\ell} - m_-) H''(u_{j,\ell}) \ge \gamma (y-m_-) \ge \frac{2\gamma}{N}(j-\tilde{\ell}) \ge \frac{\gamma}{N}(\ell-\tilde{\ell}).$$ Combining all these, we have the following whenever $j>(\tilde{\ell}+\ell)/2$.
	\begin{equation}\label{third56}
	\exp\left[-N(H(x)-H(y))\right] \le \exp\left[ -\frac{2\gamma}{N}(\ell-\tilde{\ell})(\ell-j)\right]
	\end{equation}
Combining \eqref{secbound46} and \eqref{third56}, we have the following whenever $\tilde{\ell} + K \sqrt{N\log N} \le \ell \le \ell_0$:
\begin{eqnarray}\label{fourth92}
	\E_{\ell-1}(\tau_\ell) &\le& 2C \sum_{j=\tilde{\ell}}^{\ell} \exp\left[N(H_{\beta,h,p}(y)-H_{\beta,h,p}(x))\right]\nonumber\\ &\le & \sum_{j=\tilde{\ell}}^{(\tilde{\ell}+\ell)/2} O\left(\frac{1}{N}\right) + 2C\sum_{j=(\tilde{\ell}+\ell)/2 +1}^{\ell}  \exp\left[ -\frac{2\gamma}{N}(\ell-\tilde{\ell})(\ell-j)\right]\nonumber\\&\le& O(1) + \frac{2C}{1-\exp\left[ -\frac{2\gamma}{N}(\ell-\tilde{\ell})\right]}\nonumber\\ &\le& O(1)\left[1 +\frac{N}{\ell-\tilde{\ell}}\right]
\end{eqnarray}
where the last inequality followed from the fact that $\frac{2\gamma}{N}(\ell-\tilde{\ell}) \le 2\gamma$ and there exists a constant $\alpha>0$ (depending only on $\gamma$) such that $e^{-x} < 1-\alpha x$ for all $x\in [0,2\gamma]$.

Next, let $\ell_1 \le \ell \le \ell_+ - 1$, where $\ell_1 := (i_0+\delta)N/2$ and $\ell_+ := Nm_+/2$. Note that if $j \le (\ell_2+\ell)/2$, where $\ell_2 := (i_0+\frac{\delta}{2})\frac{N}{2}$, then
$$x-y = \frac{2}{N}(\ell-j) \ge \frac{\ell-\ell_2}{N} \ge \frac{\ell_1-\ell_2}{N} = \frac{\delta}{4}~.$$ Since $H$ is strictly increasing on $[y,x]$, we have the following whenever $j \le (\ell_2+\ell)/2$:
\begin{equation}\label{ard34}
	\exp[-N(H(x)-H(y))] = \exp[-\Omega_\delta(N)]~.
\end{equation}
 So, let $j > (\ell_2+\ell)/2$. For this case, we have as before
 $$\exp\left[-N(H(x)-H(y))\right] = \exp\left[ -N(x-y)H'(\xi_{j,\ell})\right] = \exp\left[ -2(\ell-j)H'(\xi_{j,\ell})\right]$$
 for some $\xi_{j,\ell} \in [y,x]$. Also, there exists $u_{j,\ell} \in [\xi_{j,\ell}, m_+]$ such that:
 $$H'(\xi_{j,\ell}) = (\xi_{j,\ell} - m_+) H''(u_{j,\ell}) \ge \tilde{\gamma} (m_+ - x) \ge \frac{2\tilde{\gamma}}{N}(\ell_+-\ell)~.$$ 
  where $-\tilde{\gamma} <0$ is an upper bound of $H''$ on $[i_0+3\delta/4,m_+]$ (note that $y\ge i_0+3\delta/4$). 
 Combining all these, we have the following whenever $j>(\ell_2+\ell)/2$.
 \begin{equation}\label{ard3564}
 	\exp\left[-N(H(x)-H(y))\right] \le \exp\left[ -\frac{4\tilde{\gamma}}{N}(\ell_+-\ell)(\ell-j)\right]
 \end{equation}
 Combining \eqref{ard34} and \eqref{ard3564}, we have the following for $\ell_1 \le \ell \le \ell_+ - 1$:
 \begin{eqnarray}\label{bdd_inq_1}
 	\E_{\ell-1}(\tau_\ell) &\le& 2C \sum_{j=\tilde{\ell}}^{\ell} \exp\left[N(H_{\beta,h,p}(y)-H_{\beta,h,p}(x))\right]\nonumber\\ &\le & 2C\sum_{j=\tilde{\ell}}^{(\ell_2+\ell)/2} \exp[-\Omega_\delta(N)] + 2C\sum_{j=(\ell_2+\ell)/2 +1}^{\ell}  \exp\left[ -\frac{4\tilde{\gamma}}{N}(\ell_+-\ell)(\ell-j)\right]\nonumber\\&\le& o(1) + \frac{2C}{1-\exp\left[ -\frac{4\tilde{\gamma}}{N}(\ell_+-\ell)\right]}\nonumber\\ &\le& o(1)+ \frac{O(N)}{\ell_+-\ell}~.
 \end{eqnarray}

Now, suppose that $\ell_+ \le \ell \le \ell_+ + D\sqrt{N}$ for some constant $D>0$. First, note that if $\ell_+ \le j \le \ell_+ + D\sqrt{N}$, then:
\begin{equation}\label{sj21}
	\exp[-N(H(x)-H(y))] = \exp\left[-\frac{N}{2}\left((x-m_+)^2 H''(\xi_1) - (y-m_+)^2 H''(\xi_2)\right)\right] = O(1).
\end{equation} 
Next, if $j< \ell_+$, then we have:
\begin{eqnarray}\label{sj245}
	\exp[-N(H(x)-H(y))] &=& \exp\left[-N\left(\frac{1}{2}(x-m_+)^2 H''(\xi_1) - (H(y)-H(m_+))\right)\right]\nonumber\\&=& O(1)\exp[N (H(y)-H(m_+))].
\end{eqnarray}
Note that:

\[ \exp[N (H(y)-H(m_+))] =
\begin{cases} 
	O(1) & \text{if}~\ell_+ - N^{\frac{1}{2}}\log N \le j < \ell_+ \\
	\exp[-\Omega(\log^2 N)]& \text{if}~ \ell_1 \le j <  \ell_+ - N^{\frac{1}{2}}\log N\\	\exp[-\Omega(N)] & \text{if}~ j \le \ell_1.
\end{cases}
\]
This, together with \eqref{sj21} and \eqref{sj245} implies that for $\ell_+ \le \ell \le \ell_+ + D\sqrt{N}$,
\begin{eqnarray}\label{sf788}
	\E_{\ell-1}(\tau_\ell) &\le& 2C \sum_{j=\tilde{\ell}}^{\ell} \exp\left[N(H_{\beta,h,p}(y)-H_{\beta,h,p}(x))\right]\nonumber\\ &\le & O(\sqrt{N}) + O(N^\frac{1}{2}\log N) + o(1) = O(N^\frac{1}{2}\log N). 
\end{eqnarray}
Finally, suppose that $\ell_0 < \ell < \ell_1$. If $j < (i_0-2\delta)N/2$ for $\delta>0$ sufficiently small, then
$$\exp[N(H(y)-H(x))] = \exp(-\Omega_\delta(N)).$$
Otherwise, i.e. if $(i_0-2\delta)N/2 \le j \le \ell$, then:
$$\exp\left[N(H(y)-H(x))\right] = \exp\left[ N(y-x)H'(\xi_{j,\ell})\right] = \exp\left[ 2(j-\ell)H'(\xi_{j,\ell})\right]$$ for some $\xi_{j,\ell} \in [y,x]\subseteq [i_0-2\delta, i_0+\delta]$. Note that $u := \inf_{t \in [i_0-2\delta, i_0+\delta]} H'(t) > 0$ and hence,

\begin{eqnarray}\label{bdd_inq_19}
	\E_{\ell-1}(\tau_\ell) &\le& 2C \sum_{j=\tilde{\ell}}^{\ell} \exp\left[N(H_{\beta,h,p}(y)-H_{\beta,h,p}(x))\right]\nonumber\\ &\le & 2C\sum_{j=\tilde{\ell}}^{(i_0-2\delta)\frac{N}{2} - 1} \exp[-\Omega_\delta(N)] + 2C\sum_{j=(i_0-2\delta)\frac{N}{2}}^{\ell}  \exp[2(j-\ell)u]\nonumber\\&\le& o(1) + \frac{2C}{1-\exp(-2u)} =  O(1)~.
\end{eqnarray}
Combining \eqref{firststbn7}, \eqref{fourth92}, \eqref{bdd_inq_1}, \eqref{sf788} and \eqref{bdd_inq_19}, we have:
\begin{eqnarray*}
	&&\E_{\tilde{\ell}-1} [\tau_{\ell_+ + D\sqrt{N}}]\nonumber\\ &=& \sum_{\ell=\tilde{\ell}}^{\ell_+ + D\sqrt{N}} \E_{\ell-1} [\tau_\ell]\\ &=& \sum_{\ell=\tilde{\ell}}^{\tilde{\ell}+K\sqrt{N\log N}} \E_{\ell-1} [\tau_\ell] + \sum_{\ell = \tilde{\ell}+K\sqrt{N\log N}}^{\ell_0} \E_{\ell-1} [\tau_\ell] + \sum_{\ell = \ell_1}^{\ell_+-1} \E_{\ell-1} [\tau_\ell] + \sum_{\ell = \ell_+}^{\ell_+ + D\sqrt{N}} \E_{\ell-1} [\tau_\ell] + \sum_{\ell = \ell_0+1}^{\ell_1 -1} \E_{\ell-1} [\tau_\ell]\\ &=& O(N\log N).
	\end{eqnarray*}
The proof of Lemma \ref{prop:hit_from_zero} is now complete.

		\section{Technical Results}\label{appendix_1}
		In this section, we prove some technical lemmas necessary for the main proofs in the paper.
		\begin{lemma}\label{regcht1}
			A point $x\in [-1,1]$ is a fixed point of $\lambda$ if and only if it is a stationary point of $H$. Further, any such fixed point $x$ of $\lambda$ satisfies $\sgn (\lambda'(x)-1) = \sgn (H''(x))$.
		\end{lemma}
		
		\begin{proof}
			Note that $H'(x) = 0$ iff $p\beta x^{p-1} +h - \tanh^{-1}(x) =0$ iff $x= \tanh(p\beta x^{p-1}+h)$ i.e. $\lambda(x)=x$. Now, suppose that $x$ is a fixed point of $\lambda$. Then,
			$$\lambda'(x) = p(p-1)\beta x^{p-2}(1-\tanh^2(p\beta x^{p-1}+h)) = p(p-1)\beta x^{p-2}(1-x^2) = 1+(1-x^2)H''(x)~.$$ Since $|H'(\pm 1)|=\infty$, we know that $1-x^2>0$, which completes the proof of Lemma \ref{regcht1}.
		\end{proof}

		\begin{lemma}\label{special_lemma}
			At a $p$-special point $(\beta,h)$, the function $\lambda = \lambda_{\beta,h,p}$ has a unique fixed point $c^* \in (-1,1)$. Further, we have:
			$$\lambda'(c^*)=1, \quad \lambda''(c^*) = 0 \quad \text{and}\quad \lambda'''(c^*)<0~.$$ 
		\end{lemma}
		
		\begin{proof}
			To begin with, define:
			\[
			\hat{\beta}_p = \frac{1}{2(p-1)}\left(\frac{p}{p-2}\right)^{\frac{p-2}{2}},
			\quad \hat{h}_p = \tanh^{-1}\left(\sqrt{\frac{p-2}{p}}\right) - \hat{\beta}_p p \left(\frac{p-2}{p}\right)^{\frac{p-1}{2}}.
			\]
			By Lemma F.2 in \cite{Mukherjee_2021}, the only $p$-special point for odd $p$ is $(\hat{\beta}_p,\hat{h}_p)$, and the only $p$-special points for even $p$ are $(\hat{\beta}_p,\pm \hat{h}_p)$. In this case, the function $H = H_{\beta,h,p}$ has a unique global maximizer $c^* := i\sqrt{1-2/p}$, if the $p$-special point under consideration is $(\hat{\beta}_p,i\hat{h}_p)$ for $i\in \{-1,1\}$. By Lemma F.2 in \cite{Mukherjee_2021} and Lemma \ref{regcht1}, $c^*$ is a fixed point of $\lambda$ and $\lambda'(c^*)=1$. Now, note that:
			$$\lambda'(x) = p(p-1)\beta x^{p-2} (1-\lambda^2(x))$$
			and hence,
			\begin{equation}\label{pspwrc}
				\lambda''(x) = \beta p(p-1)x^{p-3}\left[(p-2)(1-\lambda^2(x)) -2x\lambda(x)\lambda'(x) \right].
			\end{equation}
			Putting $x=c^*$ in \eqref{pspwrc} and recalling that $\lambda(c^*)=c^*, \lambda'(c^*)=1$, we have $$(p-2)(1-\lambda^2(x)) -2x\lambda(x)\lambda'(x)=0,\quad\text{i.e.}\quad \lambda''(c^*) = 0.$$
			Differentiating both sides of \eqref{pspwrc} followed by straightforward calculations give:
			$$\lambda'''(c^*) = -2\beta p(p-1)(p-2)\left(1-\frac{2}{p}\right)^{\frac{p-4}{2}} < 0~.$$
			Finally, note that for odd $p$, Lemma F.2 (1) in \cite{Mukherjee_2021} says that $c^*$ is the unique stationary point of $H$, and hence by Lemma \ref{regcht1}, the unique fixed point of $\lambda$. If $p$ is even and $h>0$, then once again, it follows from the proof of Lemma F.2 (2) in \cite{Mukherjee_2021} that $H'(x)>0$ on $[0,c^*)$ and $H'(x)<0$ on $(c^*,1]$, which implies that $c^*$ is the only stationary point of $H$ on $[0,1]$. Suppose that $H''(x) >0$ for some $x\in [-1,0)$. Since $H''(-1) = -\infty$ and $H''(0) = -1$, this will imply the existence of at least two negative roots of $H''$, thereby contradicting Lemma F.2 (2) in \cite{Mukherjee_2021}, which says that $c^*$ and $-c^*$ are the only roots of $H''$. Hence, $H'' < 0$ on $[-1,0)\setminus \{-c^*\}$ and $H''(-c^*)=0$. In particular, $H$ must be strictly concave on $[-1,0)$, so any stationary point in this interval must be a global maximizer on this interval. Hence, $H'(y) <0$ for any point $y$ any point in $[-1,0)$ to the right of this (fictitious) stationary point, which clearly contradicts $H'(0)>0$. The case $h<0$ can be handled similarly, thereby proving that $c^*$ is the unique stationary point of $H$ in this case, too. The proof of Lemma \ref{special_lemma} is complete.
		\end{proof}
		
		\begin{lemma}\label{derlv}
			Suppose that $p\ge 4$ is even, $\beta > \hat{\beta}_p$, and $a_1<a_2$ are the two roots of $H_{\beta,0,p}''$ on the positive side. Suppose throughout, that $h\ge 0$. Then, the following are true.
			\begin{enumerate}
				\item[(1)] If $H_{\beta,0,p}'(a_2)\le 0$, then $H_{\beta,h,p}$ has more than one local maximizer if and only if $-H_{\beta,0,p}'(a_2) < h < -H_{\beta,0,p}'(a_1)$.
				
				\item[(2)] If $H_{\beta,0,p}'(a_2) >0$, then $H_{\beta,h,p}$ has more than one local maximizer if and only if $h < -\min\{H_{\beta,0,p}'(a_1), H_{\beta,0,p}'(-a_2)\}$.
			\end{enumerate}
		\end{lemma}
		\begin{proof}
			We highlight on the following monotonicity pattern of $H_{\beta,h,p}'$ (see the proof of Lemma F.3 in \cite{Mukherjee_2021}), which will be used crucially throughout the proof: $H_{\beta,h,p}'$ is strictly decreasing on $[-1,-a_2)$, strictly increasing on $(-a_2,-a_1)$, strictly decreasing on $(-a_1,a_1)$, strictly increasing on $(a_1,a_2)$ and strictly decreasing on $(a_2,1]$.
			
			First, consider the case $H_{\beta,0,p}'(a_2)\le 0$. Suppose that $-H_{\beta,0,p}'(a_2) < h < -H_{\beta,0,p}'(a_1)$. Then, $H_{\beta,h,p}'(0) =h>0$, $H_{\beta,h,p}'(a_1) = H_{\beta,0,p}'(a_1) + h<0$ and $H_{\beta,h,p}'(a_2) = H_{\beta,0,p}'(a_2) + h>0$. Hence, $H_{\beta,h,p}'$ has roots in $(0,a_1)$ and $(a_2,1)$, both of which are local maximizers of $H_{\beta,h,p}$, since $H_{\beta,h,p}''$ is negative on both these intervals. On the other hand, if $h\ge -H_{\beta,0,p}'(a_1)$, then $H_{\beta,h,p}'(a_1) \ge 0$. This implies that $H_{\beta,h,p}'$ can only have two roots on the positive side, one of them being strictly to the right of $a_2$, and another possibly at $a_1$, the latter being an inflection point. Also, since $H_{\beta,h,p}'(-a_2) = - H_{\beta,0,p}'(a_2) + h>0$ and $H_{\beta,h,p}'(0) = h >0$, $H_{\beta,h,p}'$ cannot have any root on the non-positive side. Thus, $H_{\beta,h,p}$ has exactly one local maximizer in this case. Similarly, if $h \le -H_{\beta,0,p}'(a_2)$, then $H_{\beta,h,p}'(a_2) \le 0$ and $H_{\beta,h,p}'<0$ everywhere else on $[a_1,1]$. This means that the only possible non-negative roots of $H_{\beta,h,p}'$ are $a_2$ and some point between $0$ and $a_1$, the former being an inflection point. Further, since $H_{\beta,h,p}'(-a_2) = - H_{\beta,0,p}'(a_2) + h \ge 0$, even if this is $0$, $-a_2$ is an inflection point, and otherwise, $H_{\beta,h,p}'$ has no negative root. So, $H_{\beta,h,p}$ has exactly one local maximizer in this case, too. This proves (1).
			
			Next, note that the \textit{only if} part of (2) has already been established in the proof of part (2) of Lemma \ref{pcritdescr}. So, suppose that $h<-\min\{H_{\beta,0,p}'(a_1), H_{\beta,0,p}'(-a_2)\}$. First, assume that $H_{\beta,0,p}'(a_1) \le H_{\beta,0,p}'(-a_2)$. Then, $h<-H_{\beta,0,p}'(a_1)$, i.e. $H_{\beta,h,p}'(a_1) < 0$. So, $H_{\beta,h,p}'$ has a root in $[0,a_1)$, which is a local maximizer of $H_{\beta,h,p}$. Moreover, $H_{\beta,h,p}'(a_2) = H_{\beta,0,p}'(a_2) + h > 0$, so $H_{\beta,h,p}$ has another local maximizer in $(a_2,1)$. Finally, assume that $H_{\beta,0,p}'(a_1) > H_{\beta,0,p}'(-a_2)$. Then, $h < -H_{\beta,0,p}'(-a_2)$ and hence, $H_{\beta,h,p}'(-a_2) <0$. So, $H_{\beta,h,p}'$ has a root in $(-1,-a_2)$, which is a local maximizer of $H_{\beta,h,p}$. Moreover, once again, $H_{\beta,h,p}'(a_2)>0$, which guarantees another local maximizer in $(a_2,1)$. This proves the \textit{if} part, and completes the proof of Lemma \ref{derlv}.
		\end{proof}

		\section{A Crude Lower Bound of the Mixing Time on $\mathfrak{B}_p$}\label{ap:apxc}
		In this section, we show how the arguments in the proof of Theorem \ref{Theorem 1} (3) can be used to give a crude lower bound on the mixing time of the Glauber dynamics when $(\beta,h)\in \mathfrak{B}_p$ (as defined in Appendix \ref{appendix_geo}).
		
		{\color{black}
			\begin{proposition}\label{crudelb}
				If $(\beta,h) \in \mathfrak{B}_p$, then $t_{\mathbf{mix}}(\epsilon) = \Omega(N^{4/3})$.
			\end{proposition}
			\begin{proof}
				To begin with, in this case, $H_{\beta,h,p}$ has a local maximizer $c^*$ and a stationary inflection point $\bar{c}$. We will only deal with the case $\bar{c} > c^*$, since the other direction will follow similarly. It follows from Theorem 3.1 (3) in \cite{Mukherjee_2021} that under $\P_{\beta,h,p}$, $N^{1/2}(\bar{X}-c^*)$  converges weakly to a non-trivial distribution as $N\rightarrow \infty$. Fix a number $K \in (0,1)$ (to be chosen later), whence there exists $A>0$ (depending on $K$), such that for all large $N$,
				\begin{equation}\label{tv816}
					\P_{\beta,h,p}(|\bar{X}-c^*| \le A N^{-1/2}) \ge K~.
				\end{equation}
				Let $e_t := c_t - \bar{c}$ and set $e_0 := 2AN^{-\gamma}$, where $0<\gamma < 1$ is a quantity to be chosen later. Define a chain $\tet$ with the same transition structure as $\et$, except at $e_0$, where the $\tet$ chain stays with probability equal to that of the $\et$ chain either going up or staying at $e_0$. The chains $\et$ and $\tet$ can be coupled in such a way, that $\et$ dominates $\tet$ stochastically under $\P_{e_0}$. Next, define:
				$$\tau := \min \{t\ge 0: \tet \le AN^{-\gamma}\}$$ and let $d_t := e_0 - \tilde{e}_{t\wedge \tau}$. Note that for $x\in (AN^{-\gamma},e_0)$, we have:
				$$\E_{e_0}(\tett |\tet = x) = \E_{e_0}(\ett |\et = x) =\frac{1}{N}\left(\lambda(\bar{c}+x)-\bar{c}-x\right)+x~.$$
				A Taylor series expansion of the last term of the above identity, together with the facts $\lambda^{'}(\bar{c}) = 1$ and $\lambda(\bar{c}) = \bar{c}$ gives:
				\begin{equation*}
					\E_{e_0}(\tett |\tet = x) = x +\frac{x^2}{2N}\lambda''(\xi) \ge  x - \frac{\alpha}{N} x^2
				\end{equation*}
				for some constant $\alpha >0$, where $\xi \in (\tilde{c}, \tilde{c}+x)$. Defining $\mathscr{A}_t:= \sigma(d_1,\ldots,d_t)$, we can now proceed exactly as in the proof of Theorem 4.2 in \cite{levin2008} to conclude that:
				\begin{equation*}\label{lcbic6}
					\E_{e_0}(d_{t+1}^2-d_t^2 |\mathscr{A}_t)  \le C_A N^{-\min\{3\gamma+1,2\}}.
				\end{equation*}
				for some constant $C_A>0$ depending on $A$. Taking expectation on both sides of \eqref{lcbic} and iterating, we have:
				$$\E_{e_0} d_t^2 \le C_A t N^{-\min\{3\gamma+1,2\}}~.$$
				Once again, following the steps in \cite{levin2008}, we have:
				$$\P_{e_0}(\tau \le t) \le \frac{C_A t}{A^2} N^{2\gamma - \min\{3\gamma+1,2\}}~.$$ Fixing another number $L\in (0,K)$ and setting $t := (A^2L/C_A) N^{\min\{3\gamma+1,2\}-2\gamma}$ above, we get:
				\begin{equation*}\label{tv82}
					\P_{e_0}(\tet \le A N^{-\gamma}) \le \P_{e_0}(\tau \le t) \le L~.   
				\end{equation*} 
				Since $c^* + A N^{-1/2} < \bar{c} + A N^{-\gamma}$ for large enough $N$, we have 
				$$\P_{e_0}(c_t \le c^* + A N^{-1/2}) \le  L~.   $$
				Combining \eqref{tv81} and \eqref{tv82}, we have: 
				$$\max_{\bm x\in \mathcal{C}_N} d_{TV}\left(\P_{\bm x}(\bm X^t \in \cdot), \P_{\beta,h,p}\right) \ge \P_{\beta,h,p}(|\bar{X} - c^*| \le A N^{-1/2}) - \P_{e_0}(|c_t - c^*| \le A N^{-1/2}) \ge K-L$$
				with $t = C_{K,L} N^{\min\{3\gamma+1,2\}-2\gamma}$ for some constant $C_{K,L}>0$.
				We can choose $K$ and $L$ such that $0<L<K<1$ and $K-L>\epsilon$ to conclude that $\tm(\epsilon) \ge C_{K,L} N^{\min\{3\gamma+1,2\}-2\gamma}$. Note that
				$$\argmax_{\gamma \in (0,1)} ~\left(\min\{3\gamma+1,2\}-2\gamma\right) = \frac{1}{3}~,$$
				and the proposition follows on taking $\gamma=\frac{1}{3}$.
			\end{proof}
		}

	\end{document}